\documentclass[a4paper, 12pt]{article}

\usepackage[utf8]{inputenc}
\usepackage[T1]{fontenc}
\usepackage{lmodern}     
\usepackage[english]{babel}
\usepackage{cite, hyphenat, hyperref}

\usepackage{amscd, amsfonts, amsmath}
\usepackage{amssymb, amsthm}
\usepackage{bbm, mathrsfs}
\usepackage{enumerate}
\usepackage{makeidx}

\usepackage{nccfoots}
\usepackage{enumerate}
\usepackage{appendix}
\usepackage[a4paper, top=2cm, bottom=2cm, left=2cm, right=2cm]{geometry}
\allowdisplaybreaks

\numberwithin{equation}{section}
\theoremstyle{plain}
\newtheorem{Lemma}{Lemma}[section]
\newtheorem{Proposition}[Lemma]{Proposition}
\newtheorem{Theorem}[Lemma]{Theorem}
\newtheorem{Corollary}[Lemma]{Corollary}

\theoremstyle{definition}
\newtheorem{Definition}[Lemma]{Definition}
\newtheorem{Example}[Lemma]{Example}
\newtheorem{Examples}[Lemma]{Examples}

\newtheorem{Remark}[Lemma]{Remark}

\begin{document}
\title{\textbf{Markovian Integral Equations}}
\author{Alexander Kalinin\footnote{Department of Mathematics, University of Mannheim, Germany. Email: {\tt AlexKalinin@gmx.de}. The author gratefully acknowledges support by Deutsche Forschungsgemeinschaft (DFG) through Research Grants SCHI/3-1 and SCHI/3-2.}}
\maketitle

\begin{abstract}
We analyze multidimensional Markovian integral equations that are formulated with a time-inhomogeneous progressive Markov process that has Borel measurable transition probabilities. In the case of a path-dependent diffusion process, the solutions to these integral equations lead to the concept of mild solutions to semilinear parabolic path-dependent partial differential equations (PPDEs). Our goal is to establish uniqueness, stability, existence, and non-extendibility of solutions among a certain class of maps. By requiring the Feller property of the Markov process, we give weak conditions under which solutions become continuous. Moreover, we provide a multidimensional Feynman-Kac formula and a one-dimensional global existence- and uniqueness result.
\end{abstract}


\section{Introduction}

Markovian integral equations arise when dealing with diffusion processes and mild solutions to semilinear parabolic partial differential equations (PDEs). This fact was utilized by Dynkin~\cite{DynkinProb, DynkinSup} to give probabilistic formulas for mild solutions via the log-Laplace functionals of superprocesses. In this context, Schied~\cite{Schied} used Markovian integral equations to solve problems of optimal stochastic control in mathematical finance. By introducing path-dependent diffusion processes, the connection of Markovian equations to PDEs can be extended to path-dependent partial differential equations (PPDEs)\footnote{For a recent analysis of PPDEs in the context of classical and viscosity solutions, we refer the reader to Peng~\cite{PengViscosity, PengBSDE}, Peng and Wang~\cite{PengWang}, Ji and Yang~\cite{JiYang}, Ekren, Keller, Touzi, and Zhang~\cite{EkrenKellerTouziZhang}, and Henri-Labordere, Tan, and Touzi~\cite{Henry-LabordereTanTouzi}.}, as verified in the companion paper~\cite{KalininSchied}. Inspired by the applications of one-dimensional Markovian equations, the aim of this paper is to construct solutions even in a multidimensional framework.

Let $S$ be a separable metrizable topological space, $T > 0$, and $\mathscr{X}=(X,(\mathscr{F}_{t})_{t\in [0,T]},\mathbb{P})$ be a consistent progressive Markov process on some measurable space $(\Omega,\mathscr{F})$ with state space $S$ that has Borel measurable transition probabilities. We consider the following \emph{multidimensional Markovian integral equation} coupled with a terminal value condition:
\begin{equation}\label{MIE}\tag{M}
\begin{split}
E_{r,x}[u(t,X_{t})] &= u(r,x) + E_{r,x}\bigg[\int_{r}^{t}f(s,X_{s},u(s,X_{s}))\,\mu(ds)\bigg],\\
u(T,x) &= g(x)
\end{split}
\end{equation}
for all $r,t\in [0,T]$ with $r\leq t$ and each $x\in S$. Here, we assume implicitly that $k\in\mathbb{N}$, $D\in\mathscr{B}(\mathbb{R}^{k})$ has non-empty interior, $f:[0,T]\times S\times D\rightarrow\mathbb{R}^{k}$ is product measurable, $\mu$ is an atomless Borel measure on $[0,T]$, and $g:S\rightarrow D$ is Borel measurable and bounded.

We first remark that for $D=\mathbb{R}^{k}$ a Picard iteration and Banach's fixed-point theorem produce existence of solutions to \eqref{MIE} locally in time. This can be found, for example, in Pazy~\cite[Theorem 6.1.4]{Pazy} when $\mathscr{X}$ is a diffusion process. Regarding existence, we will suppose more generally that $D$ is convex. By modifying analytical methods from the classical theory of ordinary differential equations (ODEs), we will derive unique non-extendible solutions to \eqref{MIE} that are admissible in an appropriate topological sense. Moreover, weak conditions ensuring the continuity of the derived solutions will be provided. In the particular case when $D=\mathbb{R}^{k}$ and $f$ is an affine map in the third variable $w\in\mathbb{R}^{k}$, we will prove a representation for solutions to \eqref{MIE}. This gives a multidimensional generalization to the Feynman-Kac formula in Dynkin~\cite[Theorem 4.1.1]{DynkinBran}.

Let us also emphasize that non-negative solutions to one-dimensional Markovian integral equations are well-studied. Namely, for $k=1$ and $D=\mathbb{R}_{+}$, solutions to \eqref{MIE} have been deduced by a Picard iteration approach. For instance, the classical references are Watanabe~\cite[Proposition 2.2]{Watanabe}, Fitzsimmons~\cite[Proposition 2.3]{Fitzsimmons}, and Iscoe~\cite[Theorem A]{Iscoe}. In these works the existence of solutions to \eqref{MIE} is used for the construction of superprocesses. Dynkin~\cite{DynkinPart, DynkinPath, DynkinBran} establishes superprocesses with probabilistic methods by means of branching particle systems, which in turn yields another existence result to our Markovian integral equations.

These treatments of \eqref{MIE} in one dimension require that the function $f$ admits a representation that is related to measure-valued branching processes. To give one of the main examples, the following case is included in~\cite{DynkinPart, DynkinPath, DynkinBran}:
\begin{equation}\label{Specific Equation}
f(t,x,w) = b_{1}(t,x)w^{\alpha_{1}} + \cdots + b_{n}(t,x)w^{\alpha_{n}}
\end{equation}
for each $(t,x,w)\in [0,T]\times S\times\mathbb{R}_{+}$, where $n\in\mathbb{N}$, $b_{1},\dots,b_{n}:[0,T]\times S\rightarrow\mathbb{R}_{+}$ are Borel measurable and bounded, and $\alpha_{1},\dots,\alpha_{n}\in [1,2]$. Here, the bound $\alpha_{i}\leq 2$ for all $i\in\{1,\dots,n\}$ is strict. However, this paper intends to derive solutions without imposing a specific form of $f$. Rather, as in the multidimensional case, we will introduce regularity conditions for $f$ with respect to the Borel measure $\mu$ like local Lipschitz $\mu$-continuity. This will allow for a more general treatment of \eqref{MIE}. In particular, our approach includes the case
\[
f(t,x,w) = a(t,x) + b_{1}(t,x)\varphi_{1}(w) + \cdots + b_{n}(t,x)\varphi_{n}(w)
\]
for all $(t,x,w)\in [0,T]\times S\times\mathbb{R}_{+}$, where $a:[0,T]\times S\rightarrow (-\infty,0]$ is Borel measurable and bounded, and $\varphi_{1},\dots,\varphi_{n}:\mathbb{R}_{+}\rightarrow\mathbb{R}_{+}$ are locally Lipschitz continuous with $\varphi_{i}(0)= 0$ for each $i\in\{1,\dots,n\}$. Hence, \eqref{Specific Equation} is also feasible if $\alpha_{i} > 2$ for some $i\in\{1,\dots,n\}$. Note that we will not restrict our attention to the case $D=\mathbb{R}_{+}$. In fact, the one-dimensional global existence and uniqueness result, we will establish, is applicable provided $D$ is a non-degenerate interval. In this connection, the same weak conditions as before grant the continuity of solutions to \eqref{MIE}.

\medskip
The paper is structured as follows. In Section~\ref{Preliminaries and main results} we set up the framework. First, in Section~\ref{Time-space Cartesian products} we consider product spaces endowed with a pseudometric and introduce several map spaces. Section~\ref{Regularity with respect to Borel measures} presents regularity conditions for multidimensional measurable maps relative to a Borel measure. In Section~\ref{Time-inhomogeneous Markov processes} we give an adjusted definition of a Markov process that is in line with the classical notion. In Section~\ref{The Markovian terminal value problem} we introduce the Markovian terminal value problem \eqref{MIE}, by defining (approximate) solutions. In Section~\ref{The main results} the main results are presented. Section~\ref{Approach to the main results} shows our approach to the main results. In Section~\ref{Comparison, stability, and growth behavior of solutions} we compare solutions, prove their stability, and also investigate their growth behavior, while in Section~\ref{Local existence in time} we construct solutions locally in time. Finally, the main results are proven in Section~\ref{Proofs of the main results}.

\section{Preliminaries and main results}\label{Preliminaries and main results}

Throughout the paper, let $S$ be a separable metrizable topological space, $T > 0$, and $\mu$ be an atomless Borel measure on $[0,T]$. We fix $k\in\mathbb{N}$ and let $\mathbbm{I}_{k}$ be the identity matrix in $\mathbb{R}^{k\times k}$. To keep notation simple, we use $|\cdot|$ for the absolute value function, the Euclidean norm on $\mathbb{R}^{k}$, and the Frobenius norm on $\mathbb{R}^{k\times k}$.

\subsection{Time-space Cartesian products}\label{Time-space Cartesian products}

We endow $[0,T]\times S$ with a pseudometric $d_{S}$ generating a topology that is coarser than the product topology, which ensures that $\mathscr{B}([0,T]\times S)\subset\mathscr{B}([0,T])\otimes\mathscr{B}(S)$, since $S$ separable. For instance, $d_{S}$ could be any product metric on $[0,T]\times S$, in which case the Borel $\sigma$-field would coincide with the product $\sigma$-field. However, the presence of a pseudometric allows us to include path processes of path-dependent diffusions as specific strong Markov processes.

Let for the moment $I$ be a non-degenerate interval in $[0,T]$ and $(E,\|\cdot\|)$ be a normed space, then we call a map $u:I\times S\rightarrow E$ \emph{consistent} if $u(r,x) = u(s,y)$  for all $(r,x),(s,y)\in I\times S$ such that $d_{S}((r,x),(s,y))=0$. Moreover, $u$ is said to be \emph{right-continuous} if for each $(r,x)\in I\times S$ and every $\varepsilon > 0$ there is $\delta > 0$ such that
\[
\|u(s,y)-u(r,x)\| < \varepsilon
\]
for all $(s,y)\in I\times S$ with $s\geq r$ and $d_{S}((s,y),(r,x)) < \delta$. Clearly, if $u$ is (right-)continuous, then it is consistent. In addition, (right-)continuity of $u$ implies that $u(\cdot,x)$ is (right-)continuous for each $x\in S$ and $u(t,\cdot)$ is continuous for all $t\in [0,T]$, which entails that $u$ is Borel measurable.

\begin{Example}\label{Pseudometric Example}
Assume that $S=C([0,T],\mathbb{R}^{d})$ for some $d\in\mathbb{N}$ and let $\rho$ be a complete metric on $S$ that is equivalent to the maximum metric, then $S$ equipped with $\rho$ is Polish. Denote each map $x\in S$ stopped at time $t\in [0,T]$ by $x^{t}\in S$, that is, $x^{t}(s) = x(s\wedge t)$ for all $s\in [0,T]$. Let
\[
d_{S}((r,x),(s,y)) = |r-s| + \rho(x^{r},y^{s})
\]
for every $(r,x),(s,y)\in [0,T]\times S$, then $[0,T]\times S$ endowed with $d_{S}$ is a separable complete pseudometric space whose topology is indeed coarser than its product topology. Further, the map $u$ is consistent if and only if it is \emph{non-anticipative} in the sense that $u(t,x) = u(t,x^{t})$ for all $(t,x)\in [0,T]\times S$. This framework is used in~\cite{EkrenKellerTouziZhang} and~\cite{KalininSchied} to deal with PPDEs.
\end{Example}

Finally, for every $D\in\mathscr{B}(E)$, we let $B(S,D)$ and $B(I\times S,D)$ denote the sets of all $D$-valued Borel measurable maps on $S$ and $I\times S$, respectively. By $B_{b}(S,D)$ and $B_{b}(I\times S,D)$ we denote the set of all bounded $g\in B(S,D)$ and $u\in B(I\times S,D)$, respectively.

\subsection{Regularity with respect to Borel measures}\label{Regularity with respect to Borel measures}

We recall that for each non-degenerate interval $I$ in $[0,T]$, a function $a\in B(I,\mathbb{R})$ is locally $\mu$-integrable if and only if $\int_{r}^{t}|a(s)|\,\mu(ds) < \infty$ for all $r,t\in I$ with $r\leq t$.

\begin{Definition}
Suppose that $I\subset [0,T]$ is a non-degenerate interval, $(E,\|\cdot\|)$ is a normed space, and $a\in B(I\times S,E)$.
\begin{enumerate}[(i)]
\item The map $a$ is called \emph{(locally) $\mu$-dominated} if there is a (locally) $\mu$-integrable $\overline{a}\in B(I,\mathbb{R}_{+})$ such that $\|a(\cdot,y)\|\leq\overline{a}$ for all $y\in S$ $\mu$-a.s.~on $I$.
\item We say that $a$ is $\mu$-suitably bounded if for each $r,t\in I$ with $r\leq t$ there is a $\mu$-null set $N\in\mathscr{B}([0,T])$ such that $\sup_{(s,y)\in (N^{c}\cap [r,t])\times S}\|a(s,y)\| < \infty$.
\end{enumerate}
\end{Definition}

By using the notation in above definition, we see immediately that the set of all $E$-valued product measurable locally $\mu$-dominated maps on $I\times S$ is a linear space that contains every $E$-valued product measurable $\mu$-suitably bounded map on $I\times S$.

\begin{Definition}
Let $f:[0,T]\times S\times D\rightarrow\mathbb{R}^{k}$ be $\mathscr{B}([0,T]\times S)\otimes\mathscr{B}(D)$-measurable.
\begin{enumerate}[(i)]
\item We call $f$ \emph{affine $\mu$-bounded} if there exist two $\mu$-dominated $a,b\in B([0,T]\times S,\mathbb{R}_{+})$ such that $|f(t,x,w)|\leq a(t,x)+ b(t,x)|w|$ for all $(t,x,w)\in [0,T]\times S\times D$. If one can take $b=0$, then $f$ is called \emph{$\mu$-bounded}.
\item We say that $f$ is \emph{locally $\mu$-bounded} at $\hat{w}\in \overline{D}$ if there is a neighborhood $W$ of $\hat{w}$ in $\overline{D}$ for which $f|([0,T]\times S\times (W\cap D))$ is $\mu$-bounded. The map $f$ is called locally $\mu$-bounded if it is locally $\mu$-bounded at each $\hat{w}\in D$.
\item Let $k=1$, then $f$ is said to be \emph{affine $\mu$-bounded from below} if $f(t,x,w) \geq -a(t,x)$ $-\, b(t,x)|w|$ for all $(t,x,w)\in [0,T]\times S\times D$ and some $\mu$-dominated $a,b\in B([0,T]\times S,\mathbb{R}_{+})$. If $b=0$ is possible, then $f$ is \emph{$\mu$-bounded from below}. Moreover, $f$ is \emph{(affine) $\mu$-bounded from above} if $-f$ is (affine) $\mu$-bounded from below.
\end{enumerate}
\end{Definition}

For a $\mathscr{B}([0,T]\times S)\otimes\mathscr{B}(D)$-measurable map $f:[0,T]\times S\times D\rightarrow\mathbb{R}^{k}$ to be locally $\mu$-bounded, it is sufficient that it is affine $\mu$-bounded. If $f$ is locally $\mu$-bounded, then the Borel measurable map $f(\cdot,\cdot,\hat{w})$ is $\mu$-dominated for each $\hat{w}\in D$. Of course, for $k=1$ the function $f$ is (affine) $\mu$-bounded if and only if it is (affine) $\mu$-bounded from below and from above. 

\begin{Definition}
Let $f:[0,T]\times S\times D\rightarrow\mathbb{R}^{k}$ be $\mathscr{B}([0,T]\times S)\otimes\mathscr{B}(D)$-measurable.
\begin{enumerate}[(i)]
\item We call $f$ \emph{Lipschitz $\mu$-continuous} if there is a $\mu$-dominated $\lambda\in B([0,T]\times S,\mathbb{R}_{+})$ satisfying $|f(t,x,w) -  f(t,x,w')| \leq \lambda(t,x)|w-w'|$ for all $(t,x)\in [0,T]\times S$ and each $w,w'\in D$.
\item We call $f$ \emph{locally Lipschitz $\mu$-continuous} at $\hat{w}\in \overline{D}$ if there is a neighborhood $W$ of $\hat{w}$ in $\overline{D}$ such that $f|([0,T]\times S\times (W\cap D))$ is Lipschitz $\mu$-continuous. The map $f$ is locally Lipschitz $\mu$-continuous if it is locally Lipschitz $\mu$-continuous at every $\hat{w}\in D$.
\end{enumerate}
\end{Definition}  
    
In what follows, the linear space of all $\mathbb{R}^{k}$-valued $\mathscr{B}([0,T]\times S)\otimes\mathscr{B}(D)$-measurable, locally $\mu$-bounded, and locally Lipschitz $\mu$-continuous maps on $[0,T]\times S\times D$ is denoted by
\begin{equation}\label{Lipschitz Space}
BC_{\mu}^{1-}([0,T]\times S\times D,\mathbb{R}^{k}).
\end{equation}
Clearly, if $f:[0,T]\times S\times D\rightarrow \mathbb{R}^{k}$ is a $\mathscr{B}([0,T]\times S)\otimes\mathscr{B}(D)$-measurable map that is locally Lipschitz $\mu$-continuous and $f(\cdot,\cdot,\hat{w})$ is $\mu$-dominated for all $\hat{w}\in D$, then $f$ is locally $\mu$-bounded. If instead $f$ is Lipschitz $\mu$-continuous and $f(\cdot,\cdot,\hat{w})$ is $\mu$-dominated for at least one $\hat{w}\in D$, then $f$ is affine $\mu$-bounded.

\begin{Examples}\label{Borel Regularity Examples}
(i) Let $a\in B([0,T]\times S,\mathbb{R}^{k})$ and $b\in B([0,T]\times S,\mathbb{R}^{k\times k})$ be $\mu$-dominated. Assume that $\varphi:D\rightarrow\mathbb{R}^{k}$ is Borel measurable and fulfills
\[
f(t,x,w) = a(t,x) + b(t,x) \varphi(w)
\]
for all $(t,x,w)\in [0,T]\times S\times D$. Then the following two assertions hold:
\begin{enumerate}[(1)]
\item $f$ is (affine) $\mu$-bounded whenever $\varphi$ is (affine) bounded. If instead $\varphi$ is locally bounded, then $f$ is locally $\mu$-bounded. For $k=1$ and $b\geq 0$, it follows that $f$ is (affine) $\mu$-bounded from below (resp.~from above) if $\varphi$ is (affine) bounded from below (resp.~from above).

\item From the (local) Lipschitz continuity of $\varphi$ the (local) Lipschitz $\mu$-continuity of $f$ follows. Thus, if $\varphi$ is locally Lipschitz continuous, then $f\in BC_{\mu}^{1-}([0,T]\times S\times D, \mathbb{R}^{k})$.
\end{enumerate}

\noindent
(ii) Let $(U,\mathscr{U})$ be a measurable space and $n$ be a kernel from $[0,T]\times S$ to $(U,\mathscr{U})$. Suppose that $\varphi:U\times D\rightarrow\mathbb{R}^{k}$ is $\mathscr{U}\otimes\mathscr{B}(D)$-measurable and $\varphi(\cdot,w)$ is $n(t,x,\cdot)$-integrable for every $(t,x,w)\in [0,T]\times S\times D$. Let $f$ be of the form
\[
f(t,x,w) = \int_{U} \varphi(u,w)\, n(t,x,du)
\]
for each $(t,x,w)\in [0,T]\times S\times D$. Then the subsequent two assertions are valid:
\begin{enumerate}[(1)]
\item $f$ is locally $\mu$-bounded if for each $\hat{w}\in D$ there are a neighborhood $W$ of $\hat{w}$ in $D$ and an $\mathscr{U}$-measurable $a:U\rightarrow [0,\infty]$ with $|\varphi(u,w)|\leq a(u)$ for all $(u,w)\in U\times W$ such that $\int_{U}a(u)\,n(\cdot,\cdot,du)$ is finite and $\mu$-dominated.

\item $f$ is locally Lipschitz $\mu$-continuous if to all $\hat{w}\in D$ there are a neighborhood $W$ of $\hat{w}$ in $D$ and an $\mathscr{U}$-measurable $\lambda:U\rightarrow [0,\infty]$ with $|\varphi(u,w) - \varphi(u,w')|\leq \lambda(u)|w-w'|$ for all $u\in U$ and each $w,w'\in W$ such that $\int_{U}\lambda(u)\,n(\cdot,\cdot,du)$ is finite and $\mu$-dominated.
\end{enumerate}
\end{Examples}

\subsection{Time-inhomogeneous Markov processes}\label{Time-inhomogeneous Markov processes}

In the sequel, let $\mathscr{X}$ be a \emph{consistent Markov process} on some measurable space $(\Omega,\mathscr{F})$ with state space $S$ and \emph{Borel measurable transition probabilities}, which is a triple $(X,(\mathscr{F}_{t})_{t\in [0,T]},\mathbb{P})$ that is composed of a process $X:[0,T]\times\Omega\rightarrow S$, a filtration $(\mathscr{F}_{t})_{t\in [0,T]}$ to which $X$ is adapted, and a set $\mathbb{P}=\{P_{r,x}\,|\,(r,x)\in [0,T]\times S\}$ of probability measures on $(\Omega,\mathscr{F})$ such that the following three conditions hold:
\begin{enumerate}[(i)]
\item $d_{S}((r,X_{r}),(r,y)) = 0$ for all $r\in [0,s]$ $P_{s,y}$-a.s.~for each $(s,y)\in [0,T]\times S$.
\item  The function $[0,t]\times S\rightarrow [0,1]$, $(s,x)\mapsto P_{s,x}(X_{t}\in B)$ is consistent and Borel measurable for all $t\in [0,T]$ and every $B\in\mathscr{B}(S)$.
\item $P_{r,x}(X_{t}\in B|\mathscr{F}_{s})= P_{s,X_{s}}(X_{t}\in B)$ $P_{r,x}$-a.s.~for all $r,s,t\in [0,T]$ with $r\leq s\leq t$, each $x\in S$, and every $B\in\mathscr{B}(S)$.
\end{enumerate}
Hence, if $d_{S}$ is a product metric, then (i) reduces to $X_{r} = y$ for all $r\in [0,s]$ $P_{s,y}$-a.s.~for each $(s,y)\in [0,T]\times S$ and we recover the classical definition of a time-inhomogeneous Markov process with Borel measurable transition probabilities. Moreover, let $\mathscr{X}$ be \emph{progressive}, that is, $X$ is progressively measurable with respect to its natural filtration and its natural backward filtration. For example, this is the case if $X$ is left- or right-continuous. 

Whenever necessary, we will require that $\mathscr{X}$ is \emph{(right-hand) Feller}, which means that the function $[0,t]\times S\rightarrow\mathbb{R}$, $(r,x)\mapsto E_{r,x}[\varphi(X_{t})]$ is (right-)continuous for all $t\in [0,T]$ and each continuous $\varphi\in B_{b}(S,\mathbb{R})$. In this case, it follows that the map
\begin{equation}\label{Feller Condition}
[0,t]\times S\rightarrow\mathbb{R}^{k},\quad (r,x)\mapsto E_{r,x}\bigg[\int_{r}^{t}\varphi(s,X_{s})\,\mu(ds)\bigg]
\end{equation}
is (right-)continuous for each $t\in [0,T]$ and every $\mu$-dominated $\varphi\in B([0,t]\times S,\mathbb{R}^{k})$ for which $\varphi(s,\cdot)$ is continuous for $\mu$-a.e.~$s\in [0,t]$, by dominated convergence.

\begin{Example}
Let the setting of Example~\ref{Pseudometric Example} hold, then $X$ is the \emph{path process} of a process $Y:[0,T]\times\Omega\rightarrow\mathbb{R}^{d}$ in the sense that $X_{t} = Y^{t}$ for all $t\in [0,T]$ if and only if
\begin{equation}\label{Path Condition}
X_{s}(\omega)|[0,r] = X_{r}(\omega)\quad\text{for all $r,s\in [0,T]$}
\end{equation}
with $r\leq s$ and each $\omega\in\Omega$. In this case, $Y$ is uniquely determined, $(\mathscr{F}_{t})_{t\in [0,T]}$-adapted, and continuous. Further, (i) is equivalent to $Y^{s} = y^{s}$ $P_{s,y}$-a.s.~for all $(s,y)\in [0,T]\times S$.

The class of non-anticipative progressive Markov processes $\mathscr{X}$ fulfilling condition \eqref{Path Condition} is used in~\cite{Kalinin} to construct path-dependent diffusion processes, which extend standard Markovian diffusions in the context of semilinear parabolic PPDEs. In particular, conditions granting the (right-hand) Feller property of $\mathscr{X}$ are provided there.
\end{Example}
 
\subsection{The Markovian terminal value problem}\label{The Markovian terminal value problem}

We let $D\in\mathscr{B}(\mathbb{R}^{k})$ have non-empty interior, $f:[0,T]\times S\times D\rightarrow\mathbb{R}^{k}$ be measurable with respect to $\mathscr{B}([0,T]\times S)\otimes \mathscr{B}(D)$, and $g\in B(S,D)$ be consistent in the sense that $g(x)=g(y)$ for all $x,y\in S$ with $d_{S}((T,x),(T,y))=0$. Let us assume initially that
\begin{equation*}
E_{r,x}[|g(X_{T})|] < \infty \quad \text{for all $(r,x)\in [0,T]\times S$.}
\end{equation*}
Further, we let $\varepsilon\in B([0,T]\times S,\mathbb{R}_{+})$ be $\mu$-dominated and define an interval $I$ in $[0,T]$ to be \emph{admissible} if it is of the form $I= (t,T]$ or $I=[t,T]$ for some $t\in [0,T)$. This allows us to introduce the \emph{Markovian terminal value problem} \eqref{MIE}, by defining $\varepsilon$-approximate solutions.

\begin{Definition}
An \emph{$\varepsilon$-approximate solution} to \eqref{MIE} on an admissible interval $I$ is a consistent map $u\in B(I\times S,D)$  for which both $|u(t,X_{t})|$ and $\int_{r}^{t}|f(s,X_{s},u(s,X_{s}))|\,\mu(ds)$ are finite and $P_{r,x}$-integrable such that
\[
\bigg|E_{r,x}[u(t,X_{t})] - u(r,x) - E_{r,x}\bigg[\int_{r}^{t}f(s,X_{s},u(s,X_{s}))\,\mu(ds)\bigg]\bigg| \leq E_{r,x}\bigg[\int_{r}^{t}\varepsilon(s,X_{s})\,\mu(ds)\bigg]
\]
and $u(T,x) = g(x)$ for all $r,t\in I$ with $r\leq t$ and each $x\in S$. Every $0$-approximate solution is called a \emph{solution}. If in addition $I=[0,T]$, then we will speak about a \emph{global} solution.
\end{Definition}

For each admissible interval $I$, it follows from the Markov property of $\mathscr{X}$ that a map $u\in B(I\times S,D)$ is a solution to \eqref{MIE} on $I$ if and only if $\int_{r}^{T}|f(s,X_{s},u(s,X_{s}))|\,\mu(ds)$ is finite and $P_{r,x}$-integrable such that
\[
u(r,x) = E_{r,x}[g(X_{T})] - E_{r,x}\bigg[\int_{r}^{T}f(s,X_{s},u(s,X_{s}))\,\mu(ds)\bigg]
\]
for all $(r,x)\in I\times S$. Note that $u$ is automatically consistent, as soon as these two conditions are valid. For our main results, we introduce admissibility and non-extendibility of solutions.
\begin{Definition}
Assume that $u$ is a solution to \eqref{MIE} on an admissible interval $I$.
\begin{enumerate}[(i)]
\item We say that $u$ is \emph{$\mu$-admissible} if for each $r\in I$ there is a $\mu$-null set $N\in\mathscr{B}([0,T])$ such that $u((N^{c}\cap [r,T])\times S)$ is relatively compact in $D$. Moreover, $u$ is called \emph{admissible} if $u([r,T]\times S)$ is in fact relatively compact in $D^{\circ}$ for all $r\in I$.
\item Let $u$ be an admissible solution to \eqref{MIE} on $I$. Then we call $u$ \emph{extendible} if there is another admissible solution $\tilde{u}$ to \eqref{MIE} on some admissible interval $\tilde{I}$ with $I\subsetneq\tilde{I}$ and $u=\tilde{u}$ on $I\times S$. Otherwise, $u$ is \emph{non-extendible} and $I$ is called a \emph{maximal interval of existence}.
\end{enumerate}
\end{Definition}

\subsection{The main results}\label{The main results}

We begin with non-extendibility and assume until the end of the paper that $g$ is bounded, as this requirement is necessary for an admissible solution to exist.

\begin{Theorem}\label{Non-Extendibility Theorem}
Let $D$ be convex, $f\in BC_{\mu}^{1-}([0,T]\times S\times D,\mathbb{R}^{k})$, and $g$ be bounded away from $\partial D$. Then there is a unique non-extendible admissible solution $u_{g}$ to \eqref{MIE} on a maximal interval of existence $I_{g}$ that is open in $[0,T]$. With $t_{g}^{-}:= \inf I_{g}$ either $I_{g} = [0,T]$ or
\begin{equation}\label{Boundary and Growth Condition}\tag{B}
\lim_{t\downarrow t_{g}^{-}}\inf_{x\in S}\min\left\{\mathrm{dist}(u_{g}(t,x),\partial D),\frac{1}{1 + |u_{g}(t,x)|}\right\} =0.
\end{equation}
Moreover, if $\mathscr{X}$ is (right-hand) Feller, $f(s,\cdot,\cdot)$ is continuous for $\mu$-a.e.~$s\in [0,T]$, and $g$ is continuous, then $u_{g}$ is (right-)continuous.
\end{Theorem}

Let us for the moment assume that the hypotheses of the theorem hold. If $u_{g}$ is bounded away from $\partial D$, that is, if $\mathrm{dist}(u_{g}(t,x),\partial D)$ $\geq \varepsilon$ for all $(t,x)\in I_{g}\times S$ and some $\varepsilon > 0$, and $I_{g}\neq [0,T]$, then from \eqref{Boundary and Growth Condition} it follows that
\begin{equation*}
\lim_{t\downarrow t_{g}^{-}}\sup_{x\in S}|u_{g}(t,x)| = \infty.
\end{equation*}
Let us instead suppose that $u_{g}$ is bounded. For instance, this occurs whenever $f$ is affine $\mu$-bounded, by Lemma~\ref{Growth Lemma}. Then the theorem says that either $u_{g}$ is a global solution or
\begin{equation}\label{Boundary Condition}
\lim_{t\downarrow t_{g}^{-}}\inf_{x\in S}\mathrm{dist}(u_{g}(t,x),\partial D)=0.
\end{equation}
In particular, if $u_{g}$ is not only bounded, but also its image $u_{g}(I_{g}\times S)$ is relatively compact in $D^{\circ}$, then $I_{g} = [0,T]$. In the case $D=\mathbb{R}^{k}$, we combine these considerations with a Picard iteration to obtain the following result, which just requires local Lipschitz $\mu$-continuity of $f$.

\begin{Proposition}\label{Global Existence and Approximation Proposition}
Let $D=\mathbb{R}^{k}$ and $f\in BC_{\mu}^{1-}([0,T]\times S\times\mathbb{R}^{k},\mathbb{R}^{k})$. Assume that $f$ is affine $\mu$-bounded, then $I_{g} = [0,T]$ and the sequence $(u_{n})_{n\in\mathbb{N}_{0}}$ in $B_{b}([0,T]\times S,\mathbb{R}^{k})$, defined recursively by $u_{0}(r,x)$ $:=E_{r,x}[g(X_{T})]$ and
\[
u_{n}(r,x):= u_{0}(r,x) - E_{r,x}\bigg[\int_{r}^{T}f(s,X_{s},u_{n-1}(s,X_{s}))\,\mu(ds)\bigg]
\]
for all $n\in\mathbb{N}$, converges uniformly to $u_{g}$, the unique global bounded solution to \eqref{MIE}.
\end{Proposition}

Let us at this place assume that $D=\mathbb{R}^{k}$ and $f$ is an affine map in $w\in\mathbb{R}^{k}$. In other words, there are two maps $a:[0,T]\times S\rightarrow\mathbb{R}^{k}$ and $b:[0,T]\times S\rightarrow\mathbb{R}^{k\times k}$ such that
\[
f(t,x,w) = a(t,x) + b(t,x)w
\]
for all $(t,x,w)\in [0,T]\times S\times\mathbb{R}^{k}$. As $a$ and $b$ are necessarily Borel measurable, we infer from Examples~\ref{Borel Regularity Examples} that $f$ is affine $\mu$-bounded and Lipschitz $\mu$-continuous as soon as $a$ and $b$ are $\mu$-dominated. Thus, we get a multidimensional Feynman-Kac formula, which for $k=1$ follows from Dynkin~\cite[Theorem 4.1.1]{DynkinBran} provided $a=0$ and $b\geq 0$.

\begin{Proposition}\label{Linear Proposition}
Let $D=\mathbb{R}^{k}$ and suppose that $f(t,x,w) = a(t,x) + b(t,x)w$ for every $(t,x,w)\in [0,T]\times S\times\mathbb{R}^{k}$ and some $\mu$-dominated $a\in B([0,T]\times S,\mathbb{R}^{k})$ and $b\in B([0,T]\times S,\mathbb{R}^{k\times k})$. Then $I_{g} = [0,T]$ and
\begin{equation}\label{Linear Representation Formula}
u_{g}(r,x) = E_{r,x}[\Sigma_{r,T}g(X_{T})]- E_{r,x}\bigg[\int_{r}^{T}\Sigma_{r,t}a(t,X_{t})\,\mu(dt)\bigg]
\end{equation}
for all $(r,x)\in [0,T]\times S$ and some map $\Sigma:[0,T]\times [0,T]\times\Omega\rightarrow\mathbb{R}^{k\times k}$, $(r,t,\omega)\mapsto\Sigma_{r,t}(\omega)$ with the following three properties:
\begin{enumerate}[(i)]
\item $\Sigma_{r,t}$ is $\sigma(X_{s}:s\in [r,t])$-measurable, $|\Sigma_{r,t}|\leq \sqrt{k}e^{\int_{r}^{t}|b(s,X_{s})|\,\mu(ds)}$, and $\Sigma(\omega)$ is continuous for all $r,t\in [0,T]$ with $r\leq t$ and each $\omega\in\Omega$.
\item $\Sigma_{r,r} = \mathbbm{I}_{k}$, $\Sigma_{r,s}\Sigma_{s,t} = \Sigma_{r,t}$, and $\Sigma_{r,t}(\omega)$ is an invertible matrix with $\Sigma_{r,t}(\omega)^{-1} = \Sigma_{t,r}(\omega)$ for all $r,s,t\in [0,T]$ and every $\omega\in\Omega$.
\item If $b(r,x)b(s,y) = b(s,y)b(r,x)$ for all $(r,x),(s,y)\in [0,T]\times S$, then $\Sigma_{r,t} = e^{-\int_{r}^{t}b(s,X_{s})\,\mu(ds)}$ for all $r,t\in [0,T]$ with $r\leq t$.
\end{enumerate}
\end{Proposition}

Clearly, if there are a $\mu$-dominated $c\in B([0,T]\times S,\mathbb{R})$ and $B\in\mathbb{R}^{k\times k}$ such that the map $b$ in above proposition is of the form $b(t,x) = c(t,x) B$ for all $(t,x)\in [0,T]\times S$, then the commutation condition in (iii) holds. Hence, we may consider an example involving trigonometric functions.

\begin{Example}
Let $k=2$ and $a = 0$. Suppose that there are a $\mu$-dominated $c\in B([0,T]\times S,\mathbb{R})$ and $\delta,\varepsilon\in\mathbb{R}\backslash\{0\}$ such that
\[
b(r,x) = c(r,x)
\begin{pmatrix}
0 & \delta \\
\varepsilon & 0
\end{pmatrix}\quad\text{for all $(r,x)\in [0,T]\times S$.}
\]
We set $\rho:=1$, if $\delta\varepsilon > 0$, and $\rho := i\in\mathbb{C}$, otherwise. Then we can write $u_{g}$ in the form
\begin{align*}
(u_{g})_{1}(r,x) &= E_{r,x}\bigg[\cosh\bigg(-\rho\sqrt{|\delta\varepsilon|}\int_{r}^{T}c(s,X_{s})\,\mu(ds)\bigg)g_{1}(X_{T})\bigg]\\
&\quad + \rho\frac{\sqrt{|\delta\varepsilon|}}{\varepsilon}E_{r,x}\bigg[\sinh\bigg(-\rho\sqrt{|\delta\varepsilon|}\int_{r}^{T}c(s,X_{s})\,\mu(ds)\bigg)g_{2}(X_{T})\bigg],\\
(u_{g})_{2}(r,x) &= \rho\frac{\sqrt{|\delta\varepsilon|}}{\delta}E_{r,x}\bigg[\sinh\bigg(-\rho\sqrt{|\delta\varepsilon|}\int_{r}^{T}c(s,X_{s})\,\mu(ds)\bigg)g_{1}(X_{T})\bigg]\\
&\quad + E_{r,x}\bigg[\cosh\bigg(-\rho\sqrt{|\delta\varepsilon|}\int_{r}^{T}c(s,X_{s})\,\mu(ds)\bigg)g_{2}(X_{T})\bigg]
\end{align*}
for all $(r,x)\in [0,T]\times S$.
\end{Example}

Let us now restrict our attention to $k=1$. While Proposition~\ref{Global Existence and Approximation Proposition} covers the case $D=\mathbb{R}$, we can also derive global solutions if $D$ is a non-degenerate interval.

\begin{Theorem}\label{One-Dimensional Global Existence Theorem}
Let $D$ be a non-degenerate interval with $\underline{d}:=\inf D$ and $\overline{d}:=\sup D$. Assume that $f\in BC_{\mu}^{1-}([0,T]\times S\times D,\mathbb{R})$ and the following two conditions hold:
\begin{enumerate}[(i)]
\item Whenever $\underline{d} > - \infty$ (resp.~$\overline{d} < \infty$), then $f$ is both locally $\mu$-bounded and locally Lipschitz $\mu$-continuous at $\underline{d}$ (resp.~$\overline{d}$) with $\lim_{w\downarrow \underline{d}} f(\cdot,x,w) \leq 0$ (resp.~$\lim_{w\uparrow \overline{d}} f(\cdot,x,w)\geq 0$) for all $x\in S$ $\mu$-a.s.
\item If $\underline{d} = - \infty$ (resp.~$\overline{d} = \infty$), then $f$ is affine $\mu$-bounded from above (resp.~from below).
\end{enumerate}
Then there is a unique global bounded solution $\overline{u}_{g}$ to \eqref{MIE} that agrees with $u_{g}$ if $g$ is bounded away from $\{\underline{d},\overline{d}\}\cap\mathbb{R}$. Moreover, if $\mathscr{X}$ is (right-hand) Feller, $f(s,\cdot,\cdot)$ is continuous for $\mu$-a.e.~$s\in [0,T]$, and $g$ is continuous, then $\overline{u}_{g}$ is (right-)continuous.
\end{Theorem}

In the case $D=\mathbb{R}_{+}$, global bounded solutions to \eqref{MIE} can be expressed via the log-Laplace functionals of superprocesses provided $f$ admits the representation required below.

\begin{Example}
Let $D=\mathbb{R}_{+}$ and $b,c\in B_{b}([0,T]\times S,\mathbb{R}_{+})$. We let $n$ be a kernel from $[0,T]\times S$ to $(0,\infty)$ for which $\int_{0}^{\infty} u\wedge u^{2}\, n(\cdot,\cdot,du)$ is bounded. Assume that $f$ is of the form
\begin{equation}\label{Superprocesses Equation}
f(t,x,w) = b(t,x)w + c(t,x)w^{2} + \int_{0}^{\infty}(e^{-u w} - 1 + u w)\,n(t,x,du)
\end{equation}
for all $(t,x,w)\in [0,T]\times S\times\mathbb{R}_{+}$, then $f\in BC_{\mu}^{1-}([0,T]\times S\times\mathbb{R}_{+},\mathbb{R}_{+})$, due to Examples~\ref{Borel Regularity Examples}.  Hence, Theorem~\ref{One-Dimensional Global Existence Theorem} applies. For instance, let for the moment $n\in\mathbb{N}$, $\alpha_{1},\dots,\alpha_{n}\in (1,2)$, and $d_{1},\dots,d_{n}\in B_{b}([0,T]\times S,\mathbb{R}_{+})$, then $f$ could admit the representation
\[
f(t,x,w) = b(t,x)w+ c(t,x)w^{2} + \sum_{i=1}^{n} d_{i}(t,x) w^{\alpha_{i}}
\]
for each $(t,x,w)\in [0,T]\times S\times\mathbb{R}_{+}$. This follows from integration by parts  and the choice $n(t,x,B) = \sum_{i=1}^{n}d_{i}(t,x)\alpha_{i}((\alpha_{i}-1)/\Gamma(2-\alpha_{i}))\int_{B}u^{-1-\alpha_{i}}\,du$ for all $(t,x)\in [0,T]\times S$ and each Borel set $B$ in $(0,\infty)$, where $\Gamma$ denotes the Gamma function.
 
In the general case \eqref{Superprocesses Equation}, Theorem 1.1 in Dynkin~\cite{DynkinPath} yields an $(\mathscr{X},\mu,f)$-superprocess, which is a consistent progressive Markov process $\mathscr{Z}= (Z,(\mathscr{G}_{t})_{t\in [0,T]},\mathbb{Q})$ with state space $\mathscr{M}_{f}(S)$, the Polish space of all finite Borel measures on $S$, such that for each $t\in (0,T]$ and every $\tilde{g}\in B_{b}(S,\mathbb{R}_{+})$, the function
\[
[0,t]\times S\rightarrow \mathbb{R}_{+},\quad (r,x)\mapsto -\log\bigg(E_{r,\delta_{x}}^{Q}\bigg[e^{-\int_{S} \tilde{g}(x) \, dZ_{t}(x)}\bigg]\bigg)
\]
is Borel measurable and a global solution to \eqref{MIE} when $T$ and $g$ are replaced by $t$ and $\tilde{g}$, respectively. Here, $\mathbb{Q}$ is of the form $\mathbb{Q}=\{Q_{r,\lambda}\,|\, (r,\lambda)\in [0,T]\times\mathscr{M}_{f}(S)\}$ and $E_{r,\delta_{x}}^{Q}$ denotes the expectation with respect to $Q_{r,\delta_{x}}$ for all $(r,x)\in [0,T]\times S$. Thus,
\[
\overline{u}_{g}(r,x)= -\log\bigg(E_{r,\delta_{x}}^{Q}\bigg[e^{-\int_{S} g(x) \, dZ_{T}(x)}\bigg]\bigg)
\]
for each $(r,x)\in [0,T]\times S$.
\end{Example}

Finally, a combination of Theorem~\ref{One-Dimensional Global Existence Theorem} with Proposition~\ref{Linear Proposition} gives the following result.

\begin{Corollary}\label{One-Dimensional Linear Corollary}
Suppose that $D$ is a non-degenerate interval with $\underline{d}:=\inf D$ and $\overline{d}:=\sup D$, and there are two $\mu$-dominated $a,b\in B([0,T]\times S,\mathbb{R})$ such that
\[
f(t,x,w) = a(t,x) + b(t,x)w\quad\text{for all $(t,x,w)\in [0,T]\times S\times D$.}
\]
Additionally, for $\underline{d} > -\infty$ (resp.~$\overline{d} < \infty$) let $a(\cdot,x) + b(\cdot,x)\underline{d}\leq 0$ (resp.~$a(\cdot,x) + b(\cdot,x)\overline{d}\geq 0$) for all $x\in S$ $\mu$-a.s. Then
\begin{equation}\label{One-dimensional Linear Representation Formula}
\overline{u}_{g}(r,x) = E_{r,x}\bigg[e^{-\int_{r}^{T}b(s,X_{s})\,\mu(ds)}g(X_{T})\bigg] - E_{r,x}\bigg[\int_{r}^{T}e^{-\int_{r}^{t}b(s,X_{s})\,\mu(ds)}a(t,X_{t})\,\mu(dt)\bigg]
\end{equation}
for every $(r,x)\in [0,T]\times S$. Furthermore, whenever $\mathscr{X}$ is (right-hand) Feller, $a(s,\cdot)$ and $b(s,\cdot)$ are continuous for $\mu$-a.e.~$s\in [0,T]$, and $g$ is continuous, then $\overline{u}_{g}$ is (right-)continuous.
\end{Corollary}

\section{Approach to the main results}\label{Approach to the main results}

\subsection{Comparison, stability, and growth behavior of solutions}\label{Comparison, stability, and growth behavior of solutions}

By using consistent boundedness and local dominance, we give a Markovian Gronwall inequality. A well-known result in this direction is provided by Dynkin~\cite[Lemma 3.2]{DynkinPart}.

\begin{Lemma}\label{Gronwall's Inequality}
Let $I$ be an admissible interval, $h\in B(S,\mathbb{R}_{+})$ be such that $E_{r,x}[|h(X_{T})|] < \infty$ for all $(r,x)\in [0,T]\times S$, and $a,b\in B(I\times S,\mathbb{R}_{+})$ be locally $\mu$-dominated. Suppose that $u\in B(I\times S,\mathbb{R}_{+})$ is $\mu$-suitably bounded and fulfills
\[
u(r,x) \leq E_{r,x}[h(X_{T})] + E_{r,x}\bigg[\int_{r}^{T}a(s,X_{s}) + b(s,X_{s})u(s,X_{s})\,\mu(ds)\bigg]
\]
for each $(r,x)\in I\times S$, then
\[
u(r,x) \leq E_{r,x}\bigg[e^{\int_{r}^{T}b(s,X_{s})\,\mu(ds)}\bigg(h(X_{T}) + \int_{r}^{T}a(s,X_{s})\,\mu(ds)\bigg)\bigg]
\]
for every $(r,x)\in I\times S$.
\end{Lemma}

\begin{proof}
It follows inductively from the Markov property of $\mathscr{X}$ and integration by parts that
\begin{align*}
u(r,x) &\leq \sum_{i=0}^{n}E_{r,x}\bigg[\frac{1}{i!}\int_{r}^{T}\bigg(b(s,X_{s})\,\mu(ds)\bigg)^{i}\bigg(h(X_{T}) + \int_{r}^{T}a(s,X_{s})\mu(ds)\bigg)\bigg]\\
&\quad + E_{r,x}\bigg[\int_{r}^{T}\bigg(\int_{r}^{t}b(s,X_{s})\,\mu(ds)\bigg)^{n}\frac{b(t,X_{t})}{n!} u(t,X_{t})\,\mu(dt)\bigg]
\end{align*}
for all $(r,x)\in I\times S$ and each $n\in\mathbb{N}$. Since $u$ is $\mu$-suitably bounded, dominated convergence yields that
\[
\lim_{n\uparrow\infty} E_{r,x}\bigg[\int_{r}^{T}\left(\int_{r}^{t}b(s,X_{s})\,\mu(ds)\right)^{n}\frac{b(t,X_{t})}{n!} u(t,X_{t})\,\mu(dt)\bigg] = 0
\]
for each $(r,x)\in I\times S$. Hence, monotone convergence gives the asserted estimate.
\end{proof}

Let us compare approximate solutions.

\begin{Lemma}\label{Comparison Lemma}
Assume that $f|([0,T]\times S\times W)$ is Lipschitz $\mu$-continuous for some set $W$ in $D$. That is, there is a $\mu$-dominated $\lambda\in B([0,T]\times S,\mathbb{R}_{+})$ such that
\[|f(t,x,w) - f(t,x,w')| \leq \lambda(t,x)|w-w'|\quad\text{for all $(t,x)\in [0,T]\times S$}
\]
and each $w,w'\in W$. Let $\varepsilon,\tilde{\varepsilon}\in B([0,T]\times S,\mathbb{R}_{+})$ be $\mu$-dominated, $\tilde{g}\in B_{b}(S,D)$ be consistent, and $I$ be an admissible interval. Then every $\varepsilon$-approximate solution $u$ to \eqref{MIE} on $I$ and each $\tilde{\varepsilon}$-approximate solution $\tilde{u}$ to \eqref{MIE} on $I$, where $g$ is replaced by $\tilde{g}$, satisfy
\[
|u-\tilde{u}|(r,x) \leq E_{r,x}\bigg[e^{\int_{r}^{T}\lambda(s,X_{s})\,\mu(ds)}\bigg(|g-\tilde{g}|(X_{T}) + \int_{r}^{T}(\varepsilon + \tilde{\varepsilon})(s,X_{s})\,\mu(ds)\bigg)\bigg]
\]
for all $(r,x)\in I\times S$ provided $u$, $\tilde{u}$ are $\mu$-suitably bounded and $u(\cdot,y),\tilde{u}(\cdot,y)\in W$ for each $y\in S$ $\mu$-a.s.~on $I$.
\end{Lemma}

\begin{proof}
The triangle inequality yields that
\begin{align*}
|u-\tilde{u}|(r,x) &\leq E_{r,x}[|g-\tilde{g}|(X_{T})] + E_{r,x}\bigg[\int_{r}^{T} (\varepsilon+\tilde{\varepsilon})(s,X_{s}) + \lambda(s,X_{s})|u-\tilde{u}|(s,X_{s})\,\mu(ds)\bigg]
\end{align*}
for each $(r,x)\in I\times S$, since $|f(s,X_{s},u(s,X_{s})) - f(s,X_{s},\tilde{u}(s,X_{s}))| \leq \lambda(s,X_{s})|u-\tilde{u}|(s,X_{s})$ for $\mu$-a.e.~$s\in [r,T]$. Hence, Lemma~\ref{Gronwall's Inequality} leads us to the asserted estimate.
\end{proof}

From the comparison we get an uniqueness result provided $f$ belongs to \eqref{Lipschitz Space}. Note that the procedure of the proof originates from Theorem 6.7 in Amann~\cite{Amann}.

\begin{Corollary}\label{Uniqueness Corollary}
Suppose that $f\in BC_{\mu}^{1-}([0,T]\times S\times D,\mathbb{R}^{k})$. Then there is at most a unique $\mu$-admissible solution to \eqref{MIE} on every admissible interval $I$.
\end{Corollary}

\begin{proof}
Suppose that $u$ and $\tilde{u}$ are two $\mu$-admissible solutions to \eqref{MIE} on $I$ and let $r\in I$. Then there is a compact set $K$ in $D$ such that $u(\cdot,y)$, $\tilde{u}(\cdot,y)\in K$ for all $y\in S$ $\mu$-a.s.~on $[r,T]$. As $K$ is compact, it follows despite of minor modifications from Proposition 6.4 in Amann~\cite{Amann} that there is a neighborhood $W$ of $K$ in $D$ such that $f|([0,T]\times S\times W)$ is Lipschitz $\mu$-continuous. Hence, $u=\tilde{u}$ on $[r,T]\times S$, by Lemma~\ref{Comparison Lemma}. The assertion follows.
\end{proof}

Now, we consider stability.

\begin{Proposition}\label{Limit Proposition}
Let $f\in BC_{\mu}^{1-}([0,T]\times S\times D,\mathbb{R}^{k})$ and $I$ be an admissible interval. For each $n\in\mathbb{N}$ let $\varepsilon_{n}\in B([0,T]\times S,\mathbb{R}_{+})$ be $\mu$-dominated, $g_{n}\in B_{b}(S,D)$ be consistent, and $u_{n}$ be an $\varepsilon_{n}$-approximate solution to \eqref{MIE} on $I$ with $g$ replaced by $g_{n}$. Assume that the following three conditions hold:
\begin{enumerate}[(i)]
\item $(g_{n})_{n\in\mathbb{N}}$ and $\big(\int_{0}^{T}\varepsilon_{n}(t,X_{t})\,\mu(dt)\big)_{n\in\mathbb{N}}$ converge uniformly to $g$ and $0$, respectively.
\item The closure of $\{u_{n}(r,x)\,|\,n\in\mathbb{N}\}$ is included in $D$ for each $(r,x)\in I\times S$.
\item For each $r\in I$ there is a compact set $K$ in $D$ such that $u_{n}(\cdot,y)\in K$ for all $n\in\mathbb{N}$ and each $y\in S$ $\mu$-a.s.~on $[r,T]$.
\end{enumerate}
Then $(u_{n})_{n\in\mathbb{N}}$ converges locally uniformly in $t\in I$ and uniformly in $x\in S$ to the unique $\mu$-admissible solution to \eqref{MIE} on $I$.
\end{Proposition}

\begin{proof}
As uniqueness is covered by Corollary~\ref{Uniqueness Corollary}, we turn directly to the existence claim. Let $r\in I$ and $K$ be a compact set $K$ in $D$ so that $u_{n}(\cdot,y)\in K$ for all $n\in\mathbb{N}$ and each $y\in S$ $\mu$-a.s.~on $[r,T]$. Then there is a neighborhood $W$ of $K$ in $D$ and a $\mu$-dominated $\lambda\in B([r,T]\times S,\mathbb{R}_{+})$ with $|f(t,x,w)-f(t,x,w')|\leq \lambda(t,x)|w-w'|$ for all $(t,x)\in [r,T]\times S$ and each $w,w'\in W$. Thus, Lemma~\ref{Comparison Lemma} ensures that
\[
|u_{m} - u_{n}|(s,x) \leq E_{s,x}\bigg[e^{\int_{s}^{T}\lambda(t,X_{t})\,\mu(dt)}\bigg(|g_{m} - g_{n}|(X_{T}) + \int_{s}^{T}(\varepsilon_{m} + \varepsilon_{n})(t,X_{t})\,\mu(dt)\bigg)\bigg]
\]
for all $m, n\in\mathbb{N}$ and every $(s,x)\in [r,T]\times S$. From (i) we infer that $(u_{n})_{n\in\mathbb{N}}$ is a uniformly Cauchy sequence on $[r,T]\times S$. As (ii) holds and $r\in I$ has been arbitrarily chosen, this shows that $(u_{n})_{n\in\mathbb{N}}$ converges locally uniformly in $t\in I$ and uniformly in $x\in S$ to some map $u\in B(I\times S, D)$.

We now check that $u$ is a $\mu$-admissible solution to \eqref{MIE} on $I$. Let as before $r\in I$ and $K$ be a compact set in $D$ with $u_{n}(\cdot,y)\in K$ for all $n\in\mathbb{N}$ and each $y\in S$ $\mu$-a.s.~on $[r,T]$, which gives $u(\cdot,y)\in K$ for all $y\in S$ $\mu$-a.s.~on $[r,T]$. Let us pick a $\mu$-dominated $\lambda\in B([r,T]\times S,\mathbb{R}_{+})$ with $|f(t,x,w) - f(t,x,w')|\leq \lambda(t,x)|w-w'|$ for all $(t,x)\in [r,T]\times S$ and every $w,w'\in K$, then
\begin{align*}
\bigg|u_{n}(s,x) &- E_{s,x}[g(X_{T})] - E_{s,x}\bigg[\int_{s}^{T}f(t,X_{t},u(t,X_{t}))\,\mu(dt)\bigg]\bigg|\\
&\leq E_{s,x}[|g_{n}-g|(X_{T})] + E_{s,x}\bigg[\int_{s}^{T}\lambda(t,X_{t})|u_{n} - u|(t,X_{t}) + \varepsilon_{n}(t,X_{t})\,\mu(dt)\bigg]
\end{align*}
for all $n\in\mathbb{N}$ and each $(s,x)\in [r,T]\times S$. This entails that $(u_{n})_{n\in\mathbb{N}}$ also converges locally uniformly in $t\in I$ and uniformly in $x\in S$ to the map
\[
I\times S\rightarrow\mathbb{R}^{k},\quad (r,x)\mapsto E_{r,x}[g(X_{T})]- E_{r,x}\bigg[\int_{r}^{T}f(s,X_{s},u(s,X_{s}))\,\mu(ds)\bigg],
\]
which proves the proposition.
\end{proof}

We conclude with a growth estimate.

\begin{Lemma}\label{Growth Lemma}
Assume that $f$ is affine $\mu$-bounded. In other words, there are two $\mu$-dominated $a,b\in B([0,T]\times S,\mathbb{R}_{+})$ with $|f(t,x,w)| \leq a(t,x) + b(t,x)|w|$ for all $(t,x,w)\in [0,T]\times S\times D$. Then every $\mu$-suitably bounded solution $u$ to \eqref{MIE} on $I$ fulfills
\begin{equation}\label{Growth Estimate}
|u(r,x)| \leq E_{r,x}\bigg[e^{\int_{r}^{T}b(s,X_{s})\,\mu(ds)}\bigg(|g(X_{T})| + \int_{r}^{T}a(s,X_{s})\,\mu(ds)\bigg)\bigg]
\end{equation}
for each $(r,x)\in I\times S$.
\end{Lemma}

\begin{proof}
We see that $|u(r,x)| \leq E_{r,x}[|g(X_{T})|]+E_{r,x}\big[\int_{r}^{T}a(s,X_{s}) + b(s,X_{s})|u(s,X_{s})|\,\mu(ds)\big]$ for every $(r,x)\in I\times S$. In consequence, Lemma~\ref{Gronwall's Inequality} gives the claimed estimate.
\end{proof}

\subsection{Local existence in time}\label{Local existence in time}

We aim to construct an approximate solution locally in time. Once this is achieved, we apply the stability result of the previous section to deduce a solution as uniform limit of a sequence of approximate solutions. This is a common approach in the classical theory of ODEs (see for instance Section 7 in Amann~\cite{Amann}). 

For each $\beta > 0$ we define $N_{\mathscr{X},b}(g)$ to be the set of all $w\in\mathbb{R}^{k}$ such that $|w - E_{r,x}[g(X_{T})]| < \beta$ for some $(r,x)\in [0,T]\times S$. Because we are dealing with the transition probabilities $\mathbb{P}$, the convexity of $D$ should be required, as the lemma below indicates.

\begin{Lemma}\label{Local Boundedness Lemma}
Let $D$ be convex and $g$ be bounded away from $\partial D$, that means, there is $\varepsilon > 0$ such that $\mathrm{dist}(g(x),\partial D)\geq \varepsilon$ for all $x\in S$. Then there exists $\beta > 0$ such that
\begin{equation}\label{Local Condition 1}
\text{$N_{\mathscr{X},\beta}(g)$ is relatively compact in $D^{\circ}$.}
\end{equation}
\end{Lemma}

\begin{proof}
Let $K$ be a compact set in $D^{\circ}$ such that $g(S)\subset K$, then $\int_{S} g(x)\,P(dx)$ belongs to the convex hull of $K$ for each probability measure $P$ on $(S,\mathscr{B}(S))$. As the convexity of $D$ entails that of $D^{\circ}$, it follows from Carath{\'e}odory's Convex Hull Theorem that along with $K$ the convex hull of $K$ is a compact set in $D^{\circ}$. Hence, there is $\beta > 0$ so that $\inf_{(r,x)\in [0,T]\times S}\mathrm{dist}(E_{r,x}[g(X_{T})],\partial D) > \beta$. Since $N_{\mathscr{X},\beta}(g)$ is simply the $\beta$-neighborhood of $\{E_{r,x}[g(X_{T})]\,|\,(r,x)\in [0,T]\times S\}$, the asserted condition \eqref{Local Condition 1} follows.
\end{proof}

Until the end of this section, let $D$ be convex, $f$ be locally $\mu$-bounded, and $g$ be bounded away from $\partial D$. Due Lemma~\ref{Local Boundedness Lemma}, we can choose $\beta > 0$ satisfying \eqref{Local Condition 1}. Let $a\in B([0,T]\times S,\mathbb{R}_{+})$ be $\mu$-dominated such that $|f(t,x,w)|\leq a(t,x)$ for all $(t,x)\in [0,T]\times S$ and each $w\in\overline{N}_{\mathscr{X},\beta}(g)$, the closure of $N_{\mathscr{X},\beta}(g)$. Then
\begin{equation}\label{Local Condition 2}
E_{r,x}\bigg[\int_{r}^{T}a(s,X_{s})\,\mu(ds)\bigg] \leq \beta
\end{equation}
for all $(r,x)\in [T-\alpha,T]\times S$ and some $\alpha\in (0,T]$. The choices of $\beta$ and $\alpha$ such that \eqref{Local Condition 1} and \eqref{Local Condition 2} hold, respectively, are used to construct a $\overline{N}_{\mathscr{X},\beta}(g)$-valued solution to \eqref{MIE} on $[T-\alpha,T]$.

\begin{Proposition}\label{Approximate Solution Proposition}
Suppose that $\varepsilon\in B([0,T]\times S,\mathbb{R}_{+})$ is $\mu$-dominated and there is $\delta > 0$ so that $|f(t,x,w) - f(t,x,w')| \leq \varepsilon(t,x)$ for all $(t,x) \in [0,T]\times S$ and each $w,w'\in\overline{N}_{\mathscr{X},\beta}(g)$ with $|w-w'|<\delta$. Then there is an $\overline{N}_{\mathscr{X},\beta}(g)$-valued $\varepsilon$-approximate solution $u$ to \eqref{MIE} on $[T-\alpha,T]$. In addition, if $\mathscr{X}$ is (right-hand) Feller, $f(s,\cdot,\cdot)$ is continuous for $\mu$-a.e.~$s\in [0,T]$, and $g$ is continuous, then $u$ is (right-)continuous.
\end{Proposition}

\begin{proof}
At first, since $a$ is $\mu$-dominated, there is $\eta\in (0,\alpha]$ such that $E_{r,x}\big[\int_{r}^{t}a(s,X_{s})\,\mu(ds)\big] < \delta$ for all $r,t\in [T-\alpha,T]$ with $r\leq t< r + \eta$ and each $x\in S$. Given $\eta$, we choose $n\in\mathbb{N}$ and $t_{0},\dots,t_{n}\in [T-\alpha,T]$ such that
\[
T-\alpha = t_{n} < \cdots  < t_{0} = T\quad\text{and}\quad\max_{i\in\{1,\dots,n\}} (t_{i} - t_{i-1}) < \eta.
\]
Starting with $u_{0}:[T-\alpha,T]\times S\rightarrow\overline{N}_{\mathscr{X},\beta}(g)$ given by $u_{0}(r,x):=E_{r,x}[g(X_{T})]$, we recursively introduce a sequence $(u_{i})_{i\in\{1,\dots,n\}}$ of consistent Borel measurable maps, by letting for each $i\in\{0,\dots,n-1\}$ the map $u_{i+1}:[t_{i+1},t_{i}]\times S\rightarrow\overline{N}_{\mathscr{X},\beta}(g)$ be defined via
\[
u_{i+1}(r,x) := E_{r,x}[u_{i}(t_{i},X_{t_{i}})] - E_{r,x}\bigg[\int_{r}^{t_{i}}f(s,X_{s},E_{s,X_{s}}[u_{i}(t_{i},X_{t_{i}})])\,\mu(ds)\bigg].
\]
It follows by induction over $i\in\{1,\dots,n\}$ that $u_{i}$ is indeed a well-defined consistent Borel measurable map taking all its values in $\overline{N}_{\mathscr{X},\beta}(g)$ such that
\begin{equation}\label{Sequence Property}
|E_{r,x}[u_{i}(t,X_{t})] - u_{i}(r,x)| \leq E_{r,x}\bigg[\int_{r}^{t} a(s,X_{s})\,\mu(ds)\bigg]
\end{equation}
for all $r,t\in [t_{i},t_{i-1}]$ with $r\leq t$ and each $x\in S$. This is an immediate consequence of the facts that $u_{i}(t_{i},x) = u_{i-1}(t_{i},x)$ and $|E_{r,x}[u_{i}(t,X_{t})] - u_{0}(r,x)|$ $\leq E_{r,x}\big[\int_{t}^{t_{0}}a(t',X_{t'})\,\mu(dt')\big]$ for each $r,t\in [t_{i},t_{i-1}]$ with $r\leq t$ and every $x\in S$.

The crucial outcome of this construction is that if we define $u:[T-\alpha,T]\times S\rightarrow\overline{N}_{\mathscr{X},\beta}(g)$ by $u(r,x) := u_{i}(r,x)$ with $i\in\{1,\dots,n\}$ such that $r\in [t_{i},t_{i-1}]$, then $u$ is an $\varepsilon$-approximate solution to \eqref{MIE} on $[T-\alpha,T]$. To see this, let $i\in\{1,\dots,n\}$, then
\begin{align*}
\bigg|E_{r,x}[u(t,X_{t})] &- u(r,x) - E_{r,x}\bigg[\int_{r}^{t}f(s,X_{s},u(s,X_{s}))\,\mu(ds)\bigg]\bigg|\\
&=\bigg|E_{r,x}\bigg[\int_{r}^{t}f(s,X_{s},E_{s,X_{s}}[u_{i-1}(t_{i-1},X_{t_{i-1}})]) - f(s,X_{s},u_{i}(s,X_{s}))\,\mu(ds)\bigg]\bigg|\\
&\leq E_{r,x}\bigg[\int_{r}^{t}\varepsilon(s,X_{s})\, \mu(ds)\bigg]
\end{align*}
for every $r,t\in [t_{i},t_{i-1}]$ with $r\leq t$ and each $x\in S$, since $u_{i-1}(t_{i-1},X_{t_{i-1}}) = u_{i}(t_{i-1},X_{t_{i-1}})$ and from $t_{i-1} - s \leq \eta$ in combination with \eqref{Sequence Property} we infer that $|E_{s,X_{s}}[u_{i}(t_{i-1},X_{t_{i-1}})] - u_{i}(s,X_{s})|< \delta$ for all $s\in [t_{i},t_{i-1}]$. Hence, the first assertion follows.

Let us now suppose that $\mathscr{X}$ is (right-hand) Feller, $f(s,\cdot,\cdot)$ is continuous for $\mu$-a.e.~$s\in [0,T]$, and $g$ is continuous. Then for each non-degenerate interval $I$ in $[0,T]$ and every right-continuous $\tilde{u}\in B(I\times S,D)$, we see readily that $f(s,\cdot,\tilde{u}(s,\cdot))$ is continuous for $\mu$-a.e.~$s\in [0,T]$. In combination with \eqref{Feller Condition}, it follows inductively that $u_{1},\dots,u_{n}$ are (right-)continuous, which yields the (right-)continuity of $u$.
\end{proof}

By constructing a suitable sequence of approximate solutions, a local existence result can be derived.

\begin{Proposition}\label{Local Existence Proposition}
Let $f\in BC_{\mu}^{1-}([0,T]\times S\times D,\mathbb{R}^{k})$, then there is a unique admissible solution $u$ to \eqref{MIE} on $[T-\alpha,T]$, which is $\overline{N}_{\mathscr{X},\beta}(g)$-valued. Moreover, if $\mathscr{X}$ is (right-hand) Feller, $f(s,\cdot,\cdot)$ is continuous for $\mu$-a.e.~$s\in [0,T]$, and $g$ is continuous, then $u$ is (right-)continuous.
\end{Proposition}

\begin{proof}
The uniqueness assertion follows directly from Corollary~\ref{Uniqueness Corollary}. To establish existence, we note that, as $\overline{N}_{\mathscr{X},\beta}(g)$ is compact, there exists a $\mu$-dominated $\lambda\in B([T-\alpha,T]\times S,\mathbb{R}_{+})$ such that
\[
|f(t,x,w) - f(t,x,w')| \leq \lambda(t,x)|w-w'|
\]
for all $(t,x)\in [T-\alpha,T]\times S$ and each $w,w'\in\overline{N}_{\mathscr{X},\beta}(g)$. Thus, Proposition~\ref{Approximate Solution Proposition} provides some $\overline{N}_{\mathscr{X},\beta}(g)$-valued $(\lambda/n)$-approximate solution $u_{n}$ to \eqref{MIE} on $[T-\alpha,T]$ for each $n\in\mathbb{N}$. Additionally, if $\mathscr{X}$ is (right-hand) Feller, $f(s,\cdot,\cdot)$ is continuous for $\mu$-a.e.~$s\in [0,T]$, and $g$ is continuous, then $u_{n}$ is (right-)continuous. 

Next, Proposition~\ref{Limit Proposition} entails that $(u_{n})_{n\in\mathbb{N}}$ converges uniformly to a $\overline{N}_{\mathscr{X},\beta}(g)$-valued solution $u$ to \eqref{MIE} on $[T-\alpha,T]$, which proves the first claim. Since the uniform limit of a sequence of $\mathbb{R}^{k}$-valued (right-)continuous maps on $[T-\alpha,T]\times S$ is again (right-)continuous, the second assertion follows directly from what we have just shown.
\end{proof}

Now, we prove a fixed-point result, which we need later on.

\begin{Lemma}\label{Fixed-Point Lemma}
Let $I$ be a compact admissible interval, $\mathscr{H}$ be a closed set in $B_{b}(I\times S,\mathbb{R}^{k})$, and $\Psi:\mathscr{H}\rightarrow\mathscr{H}$ be a map for which there is a $\mu$-dominated $\lambda\in B(I\times S,\mathbb{R}_{+})$ such that
\begin{equation}\label{Fixed-Point Condition}
|\Psi(u) - \Psi(v)|(r,x) \leq E_{r,x}\bigg[\int_{r}^{T}\lambda(s,X_{s})|u-v|(s,X_{s})\,\mu(ds)\bigg]
\end{equation}
for all $u,v\in\mathscr{H}$ and each $(r,x)\in I\times S$. Then for every $u_{0}\in\mathscr{H}$, the sequence $(u_{n})_{n\in\mathbb{N}_{0}}$, recursively given by $u_{n}:=\Psi(u_{n-1})$ for all $n\in\mathbb{N}$, converges uniformly to the unique fixed-point of $\Psi$.
\end{Lemma}

\begin{proof}
Because the uniqueness assertion can be easily inferred from Lemma~\ref{Gronwall's Inequality}, we just show that $(u_{n})_{n\in\mathbb{N}_{0}}$ converges uniformly to some fixed-point of $\Psi$. By induction,
\[
|u_{n+1} - u_{n}|(r,x) \leq E_{r,x}\bigg[\int_{r}^{T}\left(\int_{r}^{t}\lambda(s,X_{s})\,\mu(ds)\right)^{n-1}\frac{\lambda(t,X_{t})}{(n-1)!} \Delta(t,X_{t})\,\mu(dt)\bigg]
\]
for all $n\in\mathbb{N}$ and every $(r,x)\in I\times S$, where $\Delta:=|\Psi(u_{0}) - u_{0}|$. From the triangle inequality and integration by parts we obtain that
\[
|u_{m} - u_{n}|(r,x) \leq \sum_{i=n}^{m-1}\frac{1}{i!}E_{r,x}\bigg[\left(\int_{r}^{T}\lambda(s,X_{s})\,\mu(ds)\right)^{i}\bigg]\sup_{(s,y)\in [r,T]\times S}\Delta(s,y)
\]
for all $m,n\in\mathbb{N}$ with $m > n$ and each $(r,x)\in I\times S$. This shows that $(u_{n})_{n\in\mathbb{N}_{0}}$ is a uniformly Cauchy sequence. Since $\mathscr{H}$ is closed in $B_{b}(I\times S,\mathbb{R}^{k})$, it converges uniformly to some $u\in\mathscr{H}$. As $(u_{n+1})_{n\in\mathbb{N}_{0}}$ also converges uniformly to $\Psi(u)$, we conclude that $u=\Psi(u)$.
\end{proof}

Let us indicate another local existence approach.

\begin{Remark}
The set $\mathscr{H}:=B_{b}([T-\alpha,T]\times S,\overline{N}_{\mathscr{X},\beta}(g))$ is closed in $B_{b}([T-\alpha,T]\times S,\mathbb{R}^{k})$ and \eqref{Local Condition 2} guarantees that the map $\Psi:\mathscr{H}\rightarrow B([T-\alpha,T]\times S,\mathbb{R}^{k})$ defined via
\[
\Psi(u)(r,x):= E_{r,x}[g(X_{T})] - E_{r,x}\bigg[\int_{r}^{T}f(s,X_{s},u(s,X_{s}))\,\mu(ds)\bigg]
\]
maps $\mathscr{H}$ into itself. So, let $f$ be locally Lipschitz $\mu$-continuous, then there is a $\mu$-dominated $\lambda\in B([T-\alpha,T]\times S,\mathbb{R}_{+})$ satisfying \eqref{Fixed-Point Condition} for all $u,v\in\mathscr{H}$ and each $(r,x)\in [T-\alpha,T]\times S$. For this reason, Lemma~\ref{Fixed-Point Lemma} implies that $\Psi$ has a unique fixed-point $u$, which is exactly the unique admissible solution to \eqref{MIE} on $[T-\alpha,T]$ that takes all its values in $\overline{N}_{\mathscr{X},b}(g)$. 

Moreover, if $\mathscr{X}$ is (right-hand) Feller, $f(s,\cdot,\cdot)$ is continuous for $\mu$-a.e.~$s\in [0,T]$, and $g$ is continuous, then from \eqref{Feller Condition} we see that $\Psi$ preserves (right-)continuity in the sense that $\Psi(\tilde{u})$ is (right-)continuous whenever $\tilde{u}\in\mathscr{H}$ is. Thus, in this case, $u$ is (right-)continuous as uniform limit of a sequence of (right-)continuous maps in $\mathscr{H}$.
\end{Remark}

\section{Proofs of the main results}\label{Proofs of the main results}

\subsection{Proof of Theorem~\ref{Non-Extendibility Theorem}}

After having constructed solutions locally in time, we derive unique non-extendible admissible solutions and provide conditions ensuring their continuity. In this regard, the proof of Theorem 7.6 in Amann~\cite{Amann} has been be a good source for ideas.

\begin{proof}[Proof of Theorem~\ref{Non-Extendibility Theorem}]
We begin with the first claim and define $I_{g}$ to be the set consisting of $\{T\}$ and of all $t\in [0,T)$ for which \eqref{MIE} admits an admissible solution on $[t,T]$. By Proposition~\ref{Local Existence Proposition}, we have $\{T\}\subsetneq I_{g}$ and hence, $t_{g}^{-}=\inf I_{g} < T$. Let $t\in (t_{g}^{-},T]$, then there is $s\in I_{g}$ with $s < t$, which means that there is an admissible solution $u$ to \eqref{MIE} on $[s,T]$. As $u|([t,T]\times S)$ is an admissible solution to \eqref{MIE} on $[t,T]$, we get that $t\in I_{g}$. Thus, $I_{g}$ is an admissible interval.

To verify that $I_{g}$ is open in $[0,T]$, we have to show that if $I_{g}\neq [0,T]$, then $t_{g}^{-}\notin I_{g}$. On the contrary, assume that $I_{g}\neq [0,T]$, but $t_{g}^{-}\in I_{g}$. Then $t_{g}^{-} > 0$ and there is an admissible solution $u$ to \eqref{MIE} on $[t_{g}^{-},T]$. Since $u(t_{g}^{-},\cdot)$ is both bounded and bounded away from $\partial D$, Proposition~\ref{Local Existence Proposition} entails that the Markovian terminal value problem \eqref{MIE} with $T$ and $g$ replaced by $t_{g}^{-}$ and $u(t_{g}^{-},\cdot)$, respectively, has an admissible solution $v$ on $[t_{g}^{-}-\alpha,t_{g}^{-}]$ for some $\alpha\in (0,t_{g}^{-}]$. Consequently, the map $w:[t_{g}^{-}-\alpha,T]\times S\rightarrow D^{\circ}$ given by $w(r,x):= u(r,x)$, if $r\geq t_{g}^{-}$, and $w(r,x):= v(r,x)$, otherwise, is another admissible solution to \eqref{MIE} on $[t_{g}^{-}-\alpha,T]$ extending $u$ and $v$. We conclude that $t_{g}^{-}-\alpha\in I_{g}$, which contradicts the definition of $t_{g}^{-}$.

Let us now introduce the unique non-extendible admissible solution to \eqref{MIE}. We recall that if $r,t\in I_{g}$ satisfy $r\leq t$, and $u$, $v$ are two admissible solutions to \eqref{MIE} on $[r,T]$ and $[t,T]$, respectively, then $u=v$ on $[t,T]\times S$, due to Corollary~\ref{Uniqueness Corollary}. So, for each $r\in I_{g}$ we can mark the unique admissible solution to \eqref{MIE} on $[r,T]$ by $u_{r}$. Then 
\[
u_{g}:I_{g}\times S\rightarrow D^{\circ}, \quad u_{g}(r,x):= u_{r}(r,x)
\]
is the unique non-extendible admissible solution to \eqref{MIE}. In fact, if $t_{g}^{-}\in I_{g}$, which occurs if and only if $t_{g}^{-} = 0$ and $I_{g} = [0,T]$, then $u_{g}(r,x)=u_{t_{g}^{-}}(r,x)$ for all $(r,x)\in [0,T]\times S$. This in turn implies that $u_{g}$ is well-defined and a global admissible solution. Now, let instead $t_{g}^{-}\notin I_{g}$, then $I_{g} = (t_{g}^{-},T]$. In this case, we pick a strictly decreasing sequence $(t_{n})_{n\in\mathbb{N}}$ in $I_{g}$ with $\lim_{n\uparrow\infty} t_{n} = t_{g}^{-}$, then
\[
u_{g}^{-1}(B) = \bigcup_{n\in\mathbb{N}} u_{t_{n}}^{-1}(B)\in\mathscr{B}(I_{g}\times S)
\]
for all $B\in\mathscr{B}(D)$, since $u_{t_{n}}^{-1}(B)\in\mathscr{B}([t_{n},T]\times S)$ for each $n\in\mathbb{N}$. Thus, $u_{g}$ is Borel measurable. The representation $u_{g}|([r,T]\times S) = u_{r}$ for each $r\in I_{g}$ implies that $u_{g}$ is an admissible solution to \eqref{MIE} on $I_{g}$. Finally, suppose that $I$ is an admissible interval with $I_{g}\subsetneq I$ and $u$ is an admissible solution to \eqref{MIE} on $I$, then there is $t\in I$ with $t\leq t_{g}^{-}$. By the definition of $I_{g}$, we obtain that $t\in I_{g}$, which is a contradiction to $I_{g}= (t_{g}^{-},T]$. This justifies that $u_{g}$ is non-extendible.

We turn to the second claim. By way of contradiction, assume that $I_{g}\neq [0,T]$, but \eqref{Boundary and Growth Condition} fails. Then $I_{g} = (t_{g}^{-},T]$, and there are $\varepsilon\in (0,1/\sqrt{2})$ and a sequence $(t_{n})_{n\in\mathbb{N}}$ in $I_{g}$ with $\lim_{n\uparrow\infty} t_{n} = t_{g}^{-}$ such that
\[
\inf_{x\in S}\min\left\{\mathrm{dist}(u_{g}(t_{n},x),\partial D),\frac{1}{1+|u_{g}(t_{n},x)|}\right\} \geq 2\varepsilon
\]
for every $n\in\mathbb{N}$. As $D_{\eta}:=\{w\in D\,|\, \mathrm{dist}(z,\partial D)\geq \eta \text{ and } |w|\leq 1/\eta\}$ is readily seen to be a convex compact set in $D^{\circ}$ for each $\eta\in (0,2\varepsilon]$, it holds that $E_{r,x}[u_{g}(t_{n},X_{t_{n}})]\in D_{2\varepsilon}$ for all $n\in\mathbb{N}$ and each $(r,x)\in [0,t_{n}]\times S$. Let $a\in B([t_{g}^{-},T]\times S,\mathbb{R}_{+})$ be $\mu$-dominated and fulfill
\[
|f(t,x,w)|\leq a(t,x)
\]
for every $(t,x,w)\in [t_{g}^{-},T]\times S\times D_{\varepsilon}$, then there exists some $\delta\in (0,T-t_{g}^{-}]$ such that $\sup_{x\in S}E_{r,x}\big[\int_{r}^{t}a(s,X_{s})\,\mu(ds)\big]$ $< \varepsilon$ for all $r,t\in [t_{g}^{-},T]$ with $r\leq t< r + \delta$. This entails that
\begin{equation}\label{Boundary Contradiction}
\text{$u_{g}(t,S)$ is relatively compact in $D_{\varepsilon}^{\circ}$}
\end{equation}
for every $n\in\mathbb{N}$ and each $t\in (t_{n}-\delta_{n},t_{n}]$, where $\delta_{n}:=\delta\wedge (t_{n} - t_{g}^{-})$. Indeed, suppose this is false, then there is $n\in\mathbb{N}$ for which $u_{g}(t,S)$ fails to be relatively compact in $D_{\varepsilon}^{\circ}$ for at least one $t\in (t_{n}-\delta_{n},t_{n}]$. We set 
\[
s_{n}:=\sup\{t\in (t_{n}-\delta_{n},t_{n}]\,|\, \text{$u_{g}(t,S)$ is not relatively compact in $D_{\varepsilon}^{\circ}$}\},
\]
then another application of Proposition~\ref{Local Existence Proposition} shows that $u_{g}(s_{n},S)$ cannot be relatively compact in $D_{\varepsilon}^{\circ}$. In particular, $s_{n} < t_{n}$, as $u_{g}(t_{n},S) \subset D_{2\varepsilon}$. These considerations imply that
\begin{equation*}
|E_{s_{n},x}[ u_{g}(t_{n},X_{t_{n}})] - u_{g}(s_{n},x)| \leq E_{s_{n},x}\bigg[\int_{s_{n}}^{t_{n}} a(s,X_{s})\,\mu(ds)\bigg] < \varepsilon
\end{equation*}
for every $x\in S$, since $t_{n} - s_{n} < \delta_{n} \leq \delta$. From $E_{s_{n},x}[u_{g}(t_{n},X_{t_{n}})]\in D_{2\varepsilon}$ and $\varepsilon^{2} < 1/2$ it follows that $|u_{g}(s_{n},x)| < |E_{s_{n},x}[u_{g}(t_{n},X_{t_{n}})]| + \varepsilon$ $\leq 1/(2\varepsilon) + \varepsilon < 1/\varepsilon$ for each $x\in S$. Moreover,
\begin{align*}
\mathrm{dist}(u_{g}(s_{n},x),\partial D) &\geq \mathrm{dist}(E_{s_{n},x}[u_{g}(t_{n},X_{t_{n}})],\partial D) - |E_{s_{n},x}[u_{g}(t_{n},X_{t_{n}})] - u_{g}(s_{n},x)|\\
& \geq 2 \varepsilon -  |E_{s_{n},x}[u_{g}(t_{n},X_{t_{n}})] - u_{g}(s_{n},x)| > \varepsilon
\end{align*}
for all $x\in S$. In consequence, it follows that $u_{g}(s_{n},S)$ is relatively compact in $D_{\varepsilon}^{\circ}$, which is a contradiction. Therefore, condition~\eqref{Boundary Contradiction} is valid.

Next, since $\lim_{n\uparrow\infty} t_{n} = t_{g}^{-}$, there is $n_{0}\in\mathbb{N}$ such that $t_{n} - t_{g}^{-}\leq \delta$ and hence, $t_{n} - \delta_{n} = t_{g}^{-}$ for all $n\in\mathbb{N}$ with $n\geq n_{0}$. Thus, \eqref{Boundary Contradiction} leads us to
\begin{equation*}
|E_{t_{g}^{-},x}[u_{g}(r,X_{r})] - E_{t_{g}^{-},x}[u_{g}(t,X_{t})]| \leq E_{t_{g}^{-},x}\bigg[\int_{r}^{t}a(s,X_{s})\,\mu(ds)\bigg] < \varepsilon
\end{equation*}
for every $r,t\in (t_{g}^{-},t_{n_{0}}]$ with $r\leq t$ and each $x\in S$. For this reason, the map $(t_{g}^{-},T]\times S\rightarrow D^{\circ}$, $(t,x)\mapsto E_{t_{g}^{-},x}[u_{g}(t,X_{t})]$ is uniformly continuous in $t\in (t_{g}^{-},T]$, uniformly in $x\in S$. Thus, there exists a unique map $\hat{w}\in B(S,D_{\varepsilon})$ such that
\begin{equation*}
\lim_{t\downarrow t_{g}^{-}} E_{t_{g}^{-},x}[u_{g}(t,X_{t})] = \hat{w}(x), \quad \text{uniformly in $x\in S$.}
\end{equation*}
At the same time, it follows from \eqref{Boundary Contradiction} together with dominated convergence that
\begin{equation}\label{Markov Integral Limit}
\lim_{r\downarrow t_{g}^{-}} E_{t_{g}^{-},x}\bigg[\int_{r}^{T} f(s,X_{s},u_{g}(s,X_{s}))\,\mu(ds)\bigg] = E_{t_{g}^{-},x}\bigg[\int_{(t_{g}^{-},T]} f(s,X_{s},u_{g}(s,X_{s}))\,\mu(ds)\bigg]
\end{equation}
for every $x\in S$. Since the map $(t_{g}^{-},T]\times S\rightarrow\mathbb{R}^{k}$, $(r,x)\mapsto E_{t_{g}^{-},x}\big[\int_{r}^{T}f(s,X_{s},u_{g}(s,X_{s}))\,\mu(ds)\big]$ is uniformly continuous in $r\in (t_{g}^{-},T]$, uniformly in $x\in S$, the limit \eqref{Markov Integral Limit} holds in fact uniformly in $x\in S$. Thus, we define $u:[t_{g}^{-},T]\times S\rightarrow D^{\circ}$ by
\[
u(t,x):= u_{g}(t,x),\quad\text{if $t > t_{g}^{-}$},\quad\text{and}\quad u(t,x) := \hat{w}(x),\quad\text{otherwise,}
\]
then it is immediate to see that $u$ is another admissible solution to \eqref{MIE} on $[t_{g}^{-},T]$. Hence, $t_{g}^{-}\in I_{g}$, which contradicts that $I_{g}$ is open in $[0,T]$. This concludes the verification of the second claim.

At last, let $\mathscr{X}$ be (right-hand) Feller, $f(s,\cdot,\cdot)$ be continuous for $\mu$-a.e.~$s\in [0,T]$, and $g$ be continuous. We define $\hat{I}_{g}$ to be the set consisting of $\{T\}$ and of all $t\in [0,T)$ for which \eqref{MIE} admits an admissible (right-)continuous solution on $[t,T]$ and set $\hat{t}_{g}^{-}:=\inf\hat{I}_{g}$. Then Proposition~\ref{Local Existence Proposition} makes sure that $\{T\}\subsetneq\hat{I}_{g}$ and thus, $\hat{t}_{g}^{-} < T$. Using similar arguments as before, it follows that $\hat{I}_{g}$ is an admissible interval that is open in $[0,T]$. 

By Corollary~\ref{Uniqueness Corollary}, the proof is complete, once we have shown that $\hat{t}_{g}^{-} = t_{g}^{-}$. Since $\hat{t}_{g}^{-}\geq t_{g}^{-}$, let us suppose that $\hat{t}_{g}^{-} > t_{g}^{-}$. Then $\hat{I}_{g}\neq [0,T]$ and hence, $\hat{I}_{g} = (\hat{t}_{g}^{-},T]$. As $u_{g}$ must be (right-)continuous on $\hat{I}_{g}\times S$ and
\[
u_{g}(r,x) = E_{r,x}[g(X_{T})] - E_{r,x}\bigg[\int_{r}^{T}f(s,X_{s},u_{g}(s,X_{s}))\,\mu(ds)\bigg]
\]
for all $(r,x)\in [\hat{t}_{g}^{-},T]\times S$, we infer from \eqref{Feller Condition} that $u_{g}$ is in fact right-continuous on $[\hat{t}_{g}^{-},T]\times S$. For this reason, we must face the contradiction that $\hat{t}_{g}^{-}\in \hat{I}_{g}$. This completes the proof.
\end{proof}

\subsection{Proofs of Propositions~\ref{Global Existence and Approximation Proposition} and~\ref{Linear Proposition}}

\begin{proof}[Proof of Proposition~\ref{Global Existence and Approximation Proposition}]
To establish the claim, we invoke Lemma~\ref{Fixed-Point Lemma}. First, since $f$ is affine $\mu$-bounded, Lemma~\ref{Growth Lemma} implies that $u_{g}$ is bounded, and as \eqref{Boundary Condition} cannot hold, we get that $I_{g} = [0,T]$. Hence, $u_{g}$ is the unique global bounded solution to \eqref{MIE}, by Theorem~\ref{Non-Extendibility Theorem}. 

We choose two $\mu$-dominated $a,b\in B([0,T]\times S,\mathbb{R}_{+})$ such that $|f(t,x,w)|\leq a(t,x) + b(t,x)|w|$ for all $(t,x,w)\in [0,T]\times S\times\mathbb{R}^{k}$ and let $\mathscr{H}$ be the set of all $u\in B([0,T]\times S,\mathbb{R}^{k})$ satisfying \eqref{Growth Estimate} for all $(r,x)\in [0,T]\times S$. Then $\mathscr{H}$ is closed in $B_{b}([0,T]\times S,\mathbb{R}^{k})$ and $u_{0},u_{g}\in\mathscr{H}$. We pick two $\mu$-integrable $\overline{a},\overline{b}\in B([0,T],\mathbb{R}_{+})$ with $a(\cdot,y)\leq \overline{a}$ and $b(\cdot,y)\leq \overline{b}$ for all $y\in S$ $\mu$-a.s., and set
\[
c:=e^{\int_{0}^{T}\overline{b}(s)\,\mu(ds)}\bigg(\sup_{y\in S}|g(y)| + \int_{0}^{T}\overline{a}(s)\,\mu(ds)\bigg).
\]
Then each map $u\in\mathscr{H}$ satisfies $|u(r,x)|\leq c$ for each $(t,x)\in [0,T]\times S$. In addition, we introduce the mapping $\Psi:\mathscr{H}\rightarrow B_{b}([0,T]\times S,\mathbb{R}^{k})$ defined via
\[
\Psi(u)(r,x) := u_{0}(r,x) - E_{r,x}\bigg[\int_{r}^{T}f(s,X_{s},u(s,X_{s}))\,\mu(ds)\bigg],
\]
then a map $u\in\mathscr{H}$ is a global solution to \eqref{MIE} if and only if it coincides with $u_{g}$, the unique fixed-point of $\Psi$. From the Markov property of $\mathscr{X}$ and integration by parts we infer that $\Psi$ maps $\mathscr{H}$ into itself. Finally, let $\lambda\in B([0,T]\times S,\mathbb{R}_{+})$ be $\mu$-dominated such that
\[
|f(t,x,w) - f(t,x,w')|\leq \lambda(t,x)|w-w'|
\]
for every $(t,x)\in [0,T]\times S$ and each $w,w'\in\mathbb{R}^{k}$ with $|w|\vee|w'| \leq c$. This guarantees that \eqref{Fixed-Point Condition} is valid for all $u,v\in\mathscr{H}$ and each $(r,x)\in [0,T]\times S$. As this was the last condition we had to check, the claim follows from Lemma~\ref{Fixed-Point Lemma}.
\end{proof}

For the proof of Proposition~\ref{Linear Proposition} we consider an integral sequence of $\mathbb{R}^{k\times k}$-valued maps. To this end, we use the conventions that $[r,t]:=[t,r]$ and $\int_{r}^{t} \overline{b}(s)\,\mu(ds) := -\int_{t}^{r}\overline{b}(s)\,\mu(ds)$ for all $r,t\in [0,T]$ with $t< r$, each $d\in\mathbb{N}$, and every $\mu$-integrable $\overline{b}\in B([0,T],\mathbb{R}^{d\times d})$.

\begin{Lemma}\label{Linear Operator Lemma}
Assume that $b\in B([0,T]\times S,\mathbb{R}^{k\times k})$ is $\mu$-dominated. Let the sequence $(\Sigma^{(n)})_{n\in\mathbb{N}_{0}}$ of $\mathbb{R}^{k\times k}$-valued maps on $[0,T]\times [0,T]\times\Omega$ be recursively given by $\Sigma_{r,t}^{(0)}(\omega):=\mathbbm{I}_{k}$ and
\[
\Sigma_{r,t}^{(n)}(\omega):=\int_{r}^{t} b(s,X_{s}(\omega))\Sigma_{s,t}^{(n-1)}(\omega)\,\mu(ds)\quad\text{for all $n\in\mathbb{N}$.}
\]
Then $\Sigma_{r,t}^{(n)}$ is $\sigma(X_{s}:s\in [r,t])$-measurable, $|\Sigma_{r,t}^{(n)}|\leq \frac{\sqrt{k}}{n!}\big(\big|\int_{r}^{t}|b(s,X_{s})|\,\mu(ds)\big|\big)^{n}$, and $\Sigma^{(n)}(\omega)$ is continuous for all $n\in\mathbb{N}_{0}$, each $r,t\in [0,T]$, and every $\omega\in\Omega$.
\end{Lemma}

\begin{proof}
We prove the lemma by induction over $n\in\mathbb{N}_{0}$. In the initial induction step $n=0$ the assignment $\Sigma^{(0)} = \mathbbm{I}_{k}$ gives all results. Let us suppose that the claims are true for some $n\in\mathbb{N}_{0}$ and pick $r,t\in [0,T]$. Then, since $\mathscr{X}$ is progressive, the map $[r,t]\times\Omega\rightarrow\mathbb{R}^{k\times k}$, $(s,\omega)\mapsto b(s,X_{s}(\omega))\Sigma_{s,t}^{(n)}(\omega)$ is $\mathscr{B}([r,t])\otimes\sigma(X_{s}:s\in [r,t])$-measurable, and as the Frobenius norm on $\mathbb{R}^{k\times k}$ is submultiplicative,
\begin{align*}
\bigg|\int_{r}^{t}|b(s,X_{s})\Sigma_{s,t}^{(n)}|\,\mu(ds)\bigg| &\leq \sqrt{k}\bigg|\int_{r}^{t}\frac{|b(s,X_{s})|}{n!}\bigg(\bigg|\int_{s}^{t}|b(s',X_{s'})|\,\mu(ds')\bigg|\bigg)^{n}\,\mu(ds)\bigg|\\
&= \frac{\sqrt{k}}{(n+1)!}\bigg(\bigg|\int_{r}^{t}|b(s,X_{s})|\,\mu(ds)\bigg|\bigg)^{n+1}.
\end{align*}
Thus, $\Sigma_{r,t}^{(n+1)}$ is well-defined and the required estimate holds. In addition, an application of Fubini's theorem to each coordinate ensures that $\Sigma_{r,t}^{(n+1)}$ is $\sigma(X_{s}:s\in [r,t])$-measurable. 

To show that $\Sigma^{(n+1)}(\omega)$ is continuous for all $\omega\in\Omega$, let again $r,t\in [0,T]$ and $(r_{m},t_{m})_{m\in\mathbb{N}}$ be a sequence in $[0,T]\times [0,T]$ that converges to $(r,t)$, then $\lim_{m\uparrow\infty}\mathbbm{1}_{[r_{m},t_{m}]}(s)\Sigma_{s,t_{m}}^{(n)} = \mathbbm{1}_{[r,t]}(s)\Sigma_{s,t}^{(n)}$ for $\mu$-a.e.~$s\in [0,T]$. Therefore, $\lim_{m\uparrow\infty}\Sigma_{r_{m},t_{m}}^{(n+1)}(\omega) $ $= \Sigma_{r,t}^{(n+1)}(\omega)$, by dominated convergence.
\end{proof}

\begin{proof}[Proof of Proposition~\ref{Linear Proposition}]
The map $f$ is affine $\mu$-bounded and Lipschitz $\mu$-continuous. Hence, Proposition~\ref{Global Existence and Approximation Proposition} entails that the sequence $(u_{n})_{n\in\mathbb{N}_{0}}$ in $B_{b}([0,T]\times S,\mathbb{R}^{k})$, recursively given by $u_{0}(r,x) := E_{r,x}[g(X_{T})]$ and
\[
u_{n}(r,x):= u_{0}(r,x) - E_{r,x}\bigg[\int_{r}^{T}a(s,X_{s}) + b(s,X_{s})u_{n-1}(s,X_{s})\,\mu(ds)\bigg]
\]
for all $n\in\mathbb{N}$, converges uniformly to $u_{g}$, the unique global bounded solution to \eqref{MIE}. With the notation of Lemma~\ref{Linear Operator Lemma}, an induction proof shows that $u_{n}$ is of the form
\[
u_{n}(r,x) = E_{r,x}\bigg[\sum_{i=0}^{n}(-1)^{i}\Sigma_{r,T}^{(i)}g(X_{T})\bigg] - E_{r,x}\bigg[\int_{r}^{T}\sum_{i=0}^{n-1}(-1)^{i}\Sigma_{r,t}^{(i)}a(t,X_{t})\,\mu(dt)\bigg]
\]
for all $n\in\mathbb{N}$ and each $(r,x)\in [0,T]\times S$. Because $\sum_{n=0}^{\infty}|(-1)^{n}\Sigma_{r,t}^{(n)}|$ $\leq \sqrt{k} e^{|\int_{r}^{t}|b(s,X_{s})|\,\mu(ds)|}$ for every $r,t\in [0,T]$, the series mapping $\sum_{n=0}^{\infty}(-1)^{n}\Sigma^{(n)}$ converges absolutely, uniformly in $(r,t,\omega)\in [0,T]\times [0,T]\times\Omega$. Lemma~\ref{Linear Operator Lemma} together with the previous estimate imply that the limit map $\Sigma:=\sum_{n=0}^{\infty}(-1)^{n}\Sigma^{(n)}$ fulfills (i). Hence, dominated convergence yields the representation formula \eqref{Linear Representation Formula}.

Let us verify that (ii) holds as well. From $\Sigma_{r,r}^{(0)} = \mathbbm{I}_{k}$ and $\Sigma_{r,r}^{(n)} = 0$ for all $n\in\mathbb{N}$ we get that $\Sigma_{r,r}=\mathbbm{I}_{k}$ for each $r\in [0,T]$. By the Cauchy product for absolutely convergent matrix series, to verify that $\Sigma_{r,s}\Sigma_{s,t} = \Sigma_{r,t}$ for every $r,s,t\in [0,T]$, it is enough to show that
\[
\sum_{i=0}^{n}\Sigma_{r,s}^{(i)}\Sigma_{s,t}^{(n-i)} = \Sigma_{r,t}^{(n)}
\]
for all $n\in\mathbb{N}_{0}$, which follows inductively. Furthermore, from $\Sigma_{r,t}\Sigma_{t,r} = \Sigma_{r,r} = \mathbbm{I}_{k}$ we conclude that $\Sigma_{r,t}(\omega)$ is invertible and $\Sigma_{r,t}(\omega)^{-1} = \Sigma_{t,r}(\omega)$ for all $r,t\in [0,T]$ and each $\omega\in\Omega$. 

Regarding (iii), let $b$ fulfill $b(r,x)b(s,y) = b(s,y)b(r,x)$ for every $(r,x),(s,y)\in [0,T]\times S$. Then the proposition follows as soon as we have proven that
\begin{equation}\label{Matrix Sequence Representation Formula}
\Sigma_{r,t}^{(n)} = \frac{1}{n!}\bigg(\int_{r}^{t}b(s,X_{s})\,\mu(ds)\bigg)^{n}
\end{equation}
for every $n\in\mathbb{N}$ and each $r,t\in [0,T]$ with $r\leq t$. Hence, we write $S_{n}$ for the set of all permutations of $\{1,\dots,n\}$ and set $C_{n}^{\sigma}(r,t):=\{(s_{1},\dots,s_{n})\in [r,t]^{n}\,|\, s_{\sigma(1)}\leq \cdots \leq s_{\sigma(n)}\}$ for each $\sigma\in S_{n}$. From the measure transformation formula we obtain that
\begin{align*}
\int_{C_{n}^{\sigma}(r,t)} b(s_{1},X_{s_{1}})&\cdots b(s_{n},X_{s_{n}})\,d\mu^{n}(s_{1},\dots,s_{n})\\
&= \int_{C_{n}(r,t)} b(s_{1},X_{s_{1}})\cdots b(s_{n},X_{s_{n}})\,d\mu^{n}(s_{1},\dots,s_{n}) = \Sigma_{r,t}^{(n)},
\end{align*}
where $C_{n}(r,t):=\{(s_{1},\dots,s_{n})\in [r,t]^{n}\,|\,s_{1}\leq \cdots\leq s_{n}\}$. In the end, we utilize that $[r,t]^{n}$ $= \bigcup_{\sigma\in S_{n}} C_{n}^{\sigma}(r,t)$. Then the hypothesis that $\mu$ is atomless and Fubini's theorem lead to
\[
\bigg(\int_{r}^{t}b(s,X_{s})\,\mu(ds)\bigg)^{n} = \sum_{\sigma\in S_{n}} \int_{C_{n}^{\sigma}(r,t)} b(s_{1},X_{s_{1}})\cdots b(s_{n},X_{s_{n}})\,d\mu^{n}(s_{1},\dots,s_{n}) = n!\Sigma_{r,t}^{(n)}.
\]
That is, \eqref{Matrix Sequence Representation Formula} is justified and the claim follows.
\end{proof}

\subsection{Proof of Theorem~\ref{One-Dimensional Global Existence Theorem}}

We restrict our attention to $k=1$. First, we use the Feynman-Kac formula \eqref{Linear Representation Formula} to represent the difference of two solutions. This idea is essentially based on Proposition 3.1 in Schied~\cite{Schied}.

\begin{Lemma}\label{Comparison Lemma 2}
Let $f,\tilde{f}\in BC_{\mu}^{1-}([0,T]\times S\times D,\mathbb{R})$, $\tilde{g}\in B_{b}(S,D)$ be consistent, $I$ be an admissible interval, $u$ be a solution to \eqref{MIE} on $I$, and $\tilde{u}$ be a solution to \eqref{MIE} on $I$ with $\tilde{f}$ and $\tilde{g}$ instead of $f$ and $g$, respectively. Assume that $u$, $\tilde{u}$ are $\mu$-admissible and define $a,b\in B(I\times S,\mathbb{R})$ by
\[
a(r,x):= (f - \tilde{f})(r,x,\tilde{u}(r,x)),\quad\text{and}\quad b(r,x):= \frac{f(r,x,u(r,x)) - f(r,x,\tilde{u}(r,x))}{(u - \tilde{u})(r,x)},
\]
if $u(r,x)\neq \tilde{u}(r,x)$, and $b(r,x):= 0$, otherwise. Then $a,b$ are locally $\mu$-dominated and
\[
(u-\tilde{u})(r,x) = E_{r,x}\bigg[e^{-\int_{r}^{T}b(s,X_{s})\,\mu(ds)}(g-\tilde{g})(X_{T})\bigg]-E_{r,x}\bigg[\int_{r}^{T}e^{-\int_{r}^{t}b(s,X_{s})\,\mu(ds)}a(t,X_{t})\,\mu(dt)\bigg]
\]
for each $(r,x)\in I\times S$. In particular, if $f\leq\tilde{f}$ and $g\geq\tilde{g}$, then $u \geq \tilde{u}$.
\end{Lemma}

\begin{proof}
The second claim is a direct consequence of the first, since $a\leq 0$ whenever $f\leq\tilde{f}$. Thus, we merely have to prove the first assertion. To check that $a$ and $b$ are locally $\mu$-dominated, it suffices to show that for each $r\in I$ there is a $\mu$-integrable $\overline{c}\in B([r,T],\mathbb{R}_{+})$ such that
\[
|a(\cdot,y)|\vee|b(\cdot,y)| \leq \overline{c}\quad \text{for each $y\in S$}\quad \text{$\mu$-a.s.~on $[r,T]$.}
\]
This condition follows readily from the local Lipschitz $\mu$-continuity of $f$, the local $\mu$-boundedness of $\tilde{f}$, and the hypothesis that $u$, $\tilde{u}$ are $\mu$-admissible. By definition, $a(t,x) + b(t,x)(u-\tilde{u})(t,x)$ $= f(t,x,u(t,x)) - \tilde{f}(t,x,\tilde{u}(t,x))$ for each $(t,x)\in I\times S$. Hence, we let $r\in I$ and choose $a_{r},b_{r}\in B([0,T]\times S,\mathbb{R})$ so that $a_{r}(t,x)=a(t,x)$ and $b_{r}(t,x)= b(t,x)$, if $t\geq r$, and $a_{r}(t,x)$ $=b_{r}(t,x)=0$, otherwise. Then $f_{r}:[0,T]\times S\times\mathbb{R}\rightarrow\mathbb{R}$ given by
\[
f_{r}(t,x,w):=a_{r}(t,x) + b_{r}(t,x)w
\]
is affine $\mu$-bounded and Lipschitz $\mu$-continuous. In addition, the restriction of $u-\tilde{u}$ to $[r,T]\times S$ is a $\mu$-admissible solution to \eqref{MIE} with $f$ and $g$ replaced by $f_{r}$ and $g-\tilde{g}$, respectively. Thus, from Proposition~\ref{Linear Proposition} and Corollary~\ref{Uniqueness Corollary} we infer the assertion.
\end{proof}

We suppose in the sequel that $D$ is an interval, and set $\underline{d}:=\inf D$ and $\overline{d}:=\sup D$.

\begin{Lemma}\label{One-Dimensional Growth Lemma}
Let $\underline{d} > -\infty$ and $f$ be affine $\mu$-bounded from below, i.e., there are two $\mu$-dominated $a,b\in B([0,T]\times S,\mathbb{R}_{+})$ with $f(t,x,w)\geq - a(t,x) - b(t,x)|w|$ for all $(t,x,w)\in [0,T]\times S\times D$. Then every $\mu$-suitably bounded solution $u$ to \eqref{MIE} on an admissible interval $I$ fulfills
\[
u(r,x) - \underline{d} \leq E_{r,x}\bigg[e^{\int_{r}^{T}b(s,X_{s})\,\mu(ds)}\bigg(g(X_{T}) - \underline{d} +\int_{r}^{T}(a + b|\underline{d}|)(s,X_{s})\,\mu(ds)\bigg)\bigg]
\]
for all $(r,x)\in I\times S$.
\end{Lemma}

\begin{proof}
It holds that
\begin{align*}
u(r,x) - \underline{d} &\leq E_{r,x}[g(X_{T}) - \underline{d}] + E_{r,x}\bigg[\int_{r}^{T}(a+ b|\underline{d}|)(s,X_{s})\,\mu(ds)\bigg]\\
&\quad + E_{r,x}\bigg[\int_{r}^{T}\beta(s,X_{s})(u(s,X_{s}) - \underline{d})\,\mu(ds)\bigg]
\end{align*}
for each $(r,x)\in I\times S$, because $|u(s,X^{s})| \leq (u(s,X^{s}) - \underline{d}) + |\overline{d}|$ for all $s\in [r,T]$. By Lemma~\ref{Gronwall's Inequality}, the asserted estimate follows.
\end{proof}

\begin{Remark}
Suppose instead that $\overline{d} < \infty$ and $f$ is affine $\mu$-bounded from above. To obtain a similar estimate in this case, we replace $D$ by $-D=\{-w\,|\,w\in D\}$ and $f$ by the function $[0,T]\times S\times (-D)\rightarrow\mathbb{R}$, $(t,x,w)\mapsto -f(t,x,-w)$, respectively, and apply the above lemma.
\end{Remark}

Next, we study the boundary behavior of solutions. To this end, we consider only the case $\underline{d} > -\infty$, as the case $\overline{d} < \infty$ can be treated similarly, by considering above remark.

\begin{Proposition}\label{Boundary Proposition}
Let $\underline{d} > -\infty$ and $f\in BC_{\mu}^{1-}([0,T]\times S\times D,\mathbb{R})$. Suppose that $f$ is both locally $\mu$-bounded and locally Lipschitz $\mu$-continuous at $\underline{d}$ with $\lim_{w\downarrow\underline{d}}f(\cdot,x,w)\leq 0$ for all $x\in S$ $\mu$-a.s., and let one of the following two conditions hold:
\begin{enumerate}[(i)]
\item $f$ is $\mu$-bounded from above.
\item $\overline{d} = \infty$ and $f$ is affine $\mu$-bounded from below.
\end{enumerate}
Then there is $c\in (0,1]$ such that each $\mu$-admissible solution $u$ to \eqref{MIE} on an admissible interval $I$ is subject to $u(r,x) - \underline{d} \geq c( E_{r,x}[g(X_{T})] - \underline{d})$ for all $(r,x)\in I\times S$.
\end{Proposition}

\begin{proof}
Whenever $\underline{d}\notin D$, then we define the extension $\overline{f}$ of $f$ to $[0,T]\times S\times (D\cup\{\underline{d}\})$ through $\overline{f}(t,x,\underline{d}):=\lim_{w\downarrow\underline{d}} f(t,x,w)$ for all $(t,x)\in [0,T]\times S$. Otherwise, we simply set $\overline{f}:=f$, which gives $\overline{f}\in BC_{\mu}^{1-}([0,T]\times S\times (D\cup\{\underline{d}\}))$ in either case. Now, let $u$ be a $\mu$-admissible solution to \eqref{MIE} on an admissible interval $I$, then Lemma~\ref{Comparison Lemma 2} implies that $a_{u}\in B(I\times S,\mathbb{R})$ defined via $a_{u}(r,x):= (\overline{f}(r,x,u(r,x)) -\overline{f}(r,x,\underline{d}))/(u(r,x) - \underline{d})$, if $u(r,x) > \underline{d}$, and $a_{u}(r,x):=0$, otherwise, is locally $\mu$-dominated and satisfies
\[
u(r,x) - \underline{d} \geq E_{r,x}\bigg[e^{-\int_{r}^{T}a_{u}(s,X_{s})\,\mu(ds)}(g(X_{T}) - \underline{d})\bigg]
\]
for each $(r,x)\in I\times S$, since $\overline{f}(t,X_{t},\underline{d})\leq 0$ for $\mu$-a.e.~$t\in [r,T]$. We derive some $\mu$-dominated $n\in B([0,T]\times S,\mathbb{R}_{+})$ such that every $\mu$-admissible solution $u$ to \eqref{MIE} on an admissible interval $I$ satisfies $a_{u}(r,x) \leq n(r,x)$ for each $(r,x)\in I\times S$. Once this is shown, the claim follows.

So, let us at first assume that (i) holds. Then there is a $\mu$-dominated $a\in B([0,T]\times S,\mathbb{R}_{+})$ with $\overline{f}(t,x,w) \leq a(t,x)$ for each $(t,x,w)\in [0,T]\times S\times D$. As $f$ is locally Lipschitz $\mu$-continuous at $\underline{d}$, there are $\delta >0$ and a $\mu$-dominated $\lambda\in B([0,T]\times S,\mathbb{R}_{+})$ fulfilling $|\overline{f}(t,x,w) -\overline{f}(t,x,w')|$ $\leq \lambda(t,x)|w-w'|$ for every $(t,x)\in [0,T]\times S$ and all $w,w'\in [\underline{d},\underline{d} +\delta)\cap D$. Hence,
\[
a_{u}(r,x) \leq \lambda(r,x)\mathbbm{1}_{[\underline{d},\underline{d}+\delta)}(u(r,x)) + \frac{a(r,x) - \overline{f}(r,x,\underline{d})}{\delta}\mathbbm{1}_{[\underline{d}+\delta,\infty)}(u(r,x)) \leq n(r,x)
\]
for every $\mu$-admissible solution $u$ to \eqref{MIE} on an admissible interval $I$ and each $(r,x)\in I\times S$, where we have set $n :=\max\{\lambda,(a - \overline{f}(\cdot,\cdot,\underline{d}))/\delta\}$. Since $f$ locally $\mu$-bounded at $\underline{d}$, we see easily that $n$ is $\mu$-dominated, as desired.

In place of assuming that $f$ is $\mu$-bounded from above, let (ii) be true. Then Lemma~\ref{One-Dimensional Growth Lemma} yields $c >\underline{d}$ such that $u(I\times S)\subset [\underline{d},c]\cap D$ for each $\mu$-admissible solution $u$ to \eqref{MIE} on an admissible interval $I$. Because $[\underline{d},c]$ is compact, there is a $\mu$-dominated $\lambda\in B([0,T]\times S,\mathbb{R}_{+})$ such that $|\overline{f}(t,x,w) - \overline{f}(t,x,w')|$ $\leq\lambda(t,x)|w-w'|$ for all $(t,x)\in [0,T]\times S$ and each $w,w'\in [\underline{d},c]$. Hence, each $\mu$-admissible solution $u$ to \eqref{MIE} on an admissible interval $I$ fulfills $|a_{u}(r,x)|\leq n(r,x)$ for all $(r,x)\in I\times S$ with $n:=\lambda$.
\end{proof}

Eventually, we are ready to establish the one-dimensional global existence- and uniqueness result.

\begin{proof}[Proof of Theorem~\ref{One-Dimensional Global Existence Theorem}]
Let us verify the first claim. We begin with the case $\underline{d} > -\infty$ and $\overline{d} < \infty$. By using the function $[0,T]\times S\times (-D)\rightarrow\mathbb{R}$, $(t,x,w)\mapsto - f(t,x,-w)$, Proposition~\ref{Boundary Proposition} yields that $I_{\tilde{g}}=[0,T]$ for every $\tilde{g}\in B_{b}(S,(\underline{d},\overline{d}))$ that is bounded away from $\{\underline{d},\overline{d}\}$. Thus, for all $n\in\mathbb{N}$ we define
\begin{equation}\label{Limit Sequence}
g_{n}:=(g\vee(\underline{d} + (\overline{d}-\underline{d})2^{-n}))\wedge(\overline{d}-(\overline{d}-\underline{d})2^{-n}),
\end{equation}
then $g_{n}\in B_{b}(S,(\underline{d},\overline{d}))$ and $\mathrm{dist}(g_{n},\{\underline{d},\overline{d}\}) \geq (\overline{d}-\underline{d})2^{-n}$, which guarantees that $I_{g_{n}} = [0,T]$. Because $|g_{n} - g|\leq (\overline{d} - \underline{d})2^{-n}$ for all $n\in\mathbb{N}$, the sequence $(g_{n})_{n\in\mathbb{N}}$ converges uniformly to $g$. If $D\subsetneq [\underline{d},\overline{d}]$, then we let $\overline{f}$ denote the unique extension of $f$ to $[0,T]\times S\times [\underline{d},\overline{d}]$ such that
\[
\overline{f}\in BC_{\mu}^{1-}([0,T]\times S\times [\underline{d},\overline{d}]).
\]
Otherwise, we just set $\overline{f}:=f$. According to Proposition~\ref{Limit Proposition}, the sequence $(u_{g_{n}})_{n\in\mathbb{N}}$ converges uniformly to the unique global bounded solution to \eqref{MIE} with $\overline{f}$ instead of $f$, which we denote by $\overline{u}_{g}$. By uniqueness, $\overline{u}_{g}= u_{g}$ whenever $g$ is bounded away from $\{\underline{d},\overline{d}\}$. Since Proposition~\ref{Boundary Proposition} also shows that $\overline{u}_{g}$ does not attain the value $\underline{d}$ (resp.~$\overline{d}$) if the same is true for $g$, the function $\overline{u}_{g}$ is $D$-valued. Hence, $\overline{u}_{g}$ is the unique global bounded solution to \eqref{MIE}.

Let us turn to the case $\underline{d} > -\infty$ and $\overline{d} = \infty$. Lemma~\ref{One-Dimensional Growth Lemma} and Proposition~\ref{Boundary Proposition} entail that $I_{\tilde{g}} = [0,T]$ for every $\tilde{g}\in B_{b}(S,(\underline{d},\infty))$ that is bounded away from $\underline{d}$. For each $n\in\mathbb{N}$ we set
\begin{equation}\label{Limit Sequence 2}
g_{n} := g\vee (\underline{d} + 2^{-n}),
\end{equation}
then $g_{n}\in B_{b}(S,(\underline{d},\infty))$ and $\mathrm{dist}(g_{n},\underline{d}) \geq 2^{-n}$, which implies that $I_{g_{n}}= [0,T]$. In addition, $|g_{n} - g| \leq 2^{-n}$ and $g_{n}(x)-\underline{d} \leq (g(x)-\underline{d})\vee(1/2)$ for all $n\in\mathbb{N}$ and each $x\in S$. We can now infer from Lemma~\ref{One-Dimensional Growth Lemma} and Proposition~\ref{Limit Proposition} that $(u_{g_{n}})_{n\in\mathbb{N}}$ converges uniformly to the unique global bounded solution to \eqref{MIE}, denoted by $\overline{u}_{g}$. Once again, uniqueness forces $\overline{u}_{g} = u_{g}$ if $g$ is bounded away from $\underline{d}$. From Proposition~\ref{Boundary Proposition} we see that $\overline{u}_{g}$ cannot attain the value $\underline{d}$ if $g(x) > \underline{d}$ for all $x\in S$. For this reason, $\overline{u}_{g}$ is $D$-valued, which concludes the case $\underline{d} > -\infty$ and $\overline{d} = \infty$. The case $\underline{d} = -\infty$ and $\overline{d} < \infty$ is a consequence of the last case, by utilizing the familiar function $[0,T]\times S\times (-D)\rightarrow\mathbb{R}$, $(t,x,w)\mapsto -f(t,x,-w)$.

In the end, we note that for each $n\in\mathbb{N}$ the function $g_{n}$ given either by \eqref{Limit Sequence} or \eqref{Limit Sequence 2}, depending on which case occurs, is continuous if $g$ is. Hence, as the uniform limit of a sequence of real-valued (right-)continuous functions on $[0,T]\times S$ is (right-)continuous, Theorem~\ref{Non-Extendibility Theorem} implies the second assertion.
\end{proof}

\begin{proof}[Proof of Corollary~\ref{One-Dimensional Linear Corollary}]
At first, Theorem~\ref{One-Dimensional Global Existence Theorem} entails that \eqref{MIE} admits the unique global bounded solution $\overline{u}_{g}$, which is (right-)continuous if $\mathscr{X}$ is (right-hand) Feller, $a(s,\cdot)$ and $b(s,\cdot)$ are continuous for $\mu$-a.e.~$s\in [0,T]$, and $g$ is continuous. Let us set
\[
\overline{f}(t,x,w) := a(t,x) + b(t,x)w\quad\text{for all $(t,x,w)\in [0,T]\times S\times\mathbb{R}$,}
\]
then Proposition~\ref{Linear Proposition} implies that the unique global bounded solution $\tilde{u}_{g}$ to \eqref{MIE} with $f$ replaced by $\overline{f}$ admits the required representation \eqref{One-dimensional Linear Representation Formula}. However, $\overline{u}_{g}$ is also a global bounded solution to \eqref{MIE} when $f$ is replaced by $\overline{f}$. Uniqueness gives $\overline{u}_{g} = \tilde{u}_{g}$.
\end{proof}

\noindent
{\bf Acknowledgments:} the author wishes to thank his supervisor Alexander Schied and his colleague Dimitri Schwab for helpful suggestions during the preparation of the paper.

\let\OLDthebibliography\thebibliography
\renewcommand\thebibliography[1]{
  \OLDthebibliography{#1}
  \setlength{\parskip}{0pt}
  \setlength{\itemsep}{2pt }}

\bibliographystyle{plain}

\begin{thebibliography}{10}
\bibitem{Amann}
Herbert Amann.
\newblock {\em Ordinary differential equations}, volume~13 of {\em De Gruyter
  Studies in Mathematics}.
\newblock Walter de Gruyter \& Co., Berlin, 1990.
\newblock An introduction to nonlinear analysis, Translated from the German by
  Gerhard Metzen.

\bibitem{DynkinPart}
E.~B. Dynkin.
\newblock Branching particle systems and superprocesses.
\newblock {\em Ann. Probab.}, 19(3):1157--1194, 1991.

\bibitem{DynkinPath}
E.~B. Dynkin.
\newblock Path processes and historical superprocesses.
\newblock {\em Probab. Theory Related Fields}, 90(1):1--36, 1991.

\bibitem{DynkinProb}
E.~B. Dynkin.
\newblock A probabilistic approach to one class of nonlinear differential
  equations.
\newblock {\em Probab. Theory Related Fields}, 89(1):89--115, 1991.

\bibitem{DynkinSup}
E.~B. Dynkin.
\newblock Superdiffusions and parabolic nonlinear differential equations.
\newblock {\em Ann. Probab.}, 20(2):942--962, 1992.

\bibitem{DynkinBran}
Eugene~B. Dynkin.
\newblock {\em An introduction to branching measure-valued processes}, volume~6
  of {\em CRM Monograph Series}.
\newblock American Mathematical Society, Providence, RI, 1994.

\bibitem{EkrenKellerTouziZhang}
Ibrahim Ekren, Christian Keller, Nizar Touzi, and Jianfeng Zhang.
\newblock On viscosity solutions of path dependent {PDE}s.
\newblock {\em Ann. Probab.}, 42(1):204--236, 2014.

\bibitem{Fitzsimmons}
P.~J. Fitzsimmons.
\newblock Construction and regularity of measure-valued {M}arkov branching
  processes.
\newblock {\em Israel J. Math.}, 64(3):337--361 (1989), 1988.

\bibitem{Henry-LabordereTanTouzi}
Pierre Henry-Labord\`ere, Xiaolu Tan, and Nizar Touzi.
\newblock A numerical algorithm for a class of {BSDE}s via the branching
  process.
\newblock {\em Stochastic Process. Appl.}, 124(2):1112--1140, 2014.

\bibitem{Iscoe}
I.~Iscoe.
\newblock A weighted occupation time for a class of measure-valued branching
  processes.
\newblock {\em Probab. Theory Relat. Fields}, 71(1):85--116, 1986.

\bibitem{JiYang}
Shaolin Ji and Shuzhen Yang.
\newblock Classical solutions of path-dependent {PDE}s and functional
  forward-backward stochastic systems.
\newblock {\em Math. Probl. Eng.}, 2013.
\newblock Article ID 423101.

\bibitem{Kalinin}
Alexander Kalinin.
\newblock Path-dependent diffusion processes.
\newblock {\em Preprint}, 2017.

\bibitem{KalininSchied}
Alexander Kalinin and Alexander Schied.
\newblock Mild and viscosity solutions to semilinear parabolic path-dependent
  {PDE}s.
\newblock {\em arXiv preprint 1611.08318v2}, 2017.

\bibitem{Pazy}
A.~Pazy.
\newblock {\em Semigroups of linear operators and applications to partial
  differential equations}, volume~44 of {\em Applied Mathematical Sciences}.
\newblock Springer-Verlag, New York, 1983.

\bibitem{PengBSDE}
Shige Peng.
\newblock Backward stochastic differential equation, nonlinear expectation and
  their applications.
\newblock In {\em Proceedings of the {I}nternational {C}ongress of
  {M}athematicians. {V}olume {I}}, pages 393--432. Hindustan Book Agency, New
  Delhi, 2010.

\bibitem{PengViscosity}
Shige Peng.
\newblock Note on viscosity solution of path-dependent {PDE} and
  {G}-martingales.
\newblock {\em arXiv preprint 1106.1144v2}, 2012.

\bibitem{PengWang}
ShiGe Peng and FaLei Wang.
\newblock B{SDE}, path-dependent {PDE} and nonlinear {F}eynman-{K}ac formula.
\newblock {\em Sci. China Math.}, 59(1):19--36, 2016.

\bibitem{Schied}
Alexander Schied.
\newblock A control problem with fuel constraint and {D}awson-{W}atanabe
  superprocesses.
\newblock {\em Ann. Appl. Probab.}, 23(6):2472--2499, 2013.

\bibitem{Watanabe}
Shinzo Watanabe.
\newblock A limit theorem of branching processes and continuous state branching
  processes.
\newblock {\em J. Math. Kyoto Univ.}, 8:141--167, 1968.
\end{thebibliography}

\end{document}